\documentclass{article}
\usepackage[utf8]{inputenc}

\usepackage[margin=1.25in,marginparwidth=1.8in]{geometry}

\usepackage[backend=biber,style=alphabetic,maxcitenames=99,maxbibnames=99,firstinits=true,url=false,backref=true]{biblatex}
\DefineBibliographyStrings{english}{%
  backrefpage = {cited on page},%
  backrefpages = {cited on pages}%
}

\usepackage{array}
\newcolumntype{X}{>{\centering\arraybackslash$} p{0.7cm} <{$}}

\usepackage{calc}
\usepackage{tikz}
\usetikzlibrary{calc}

\usepackage{amsmath,amssymb,amsthm}
\allowdisplaybreaks
\usepackage{bm}

\newtheorem{theorem}{Theorem}
\newtheorem{lemma}[theorem]{Lemma}
\newtheorem{corollary}[theorem]{Corollary}
\newtheorem{proposition}[theorem]{Proposition}
\newtheorem{assumption}[theorem]{Assumption}
\newtheorem{remark}[theorem]{Remark}

\theoremstyle{definition}
\newtheorem{definition}[theorem]{Definition}

\numberwithin{equation}{section}
\numberwithin{theorem}{section}

\usepackage{thm-restate}

\usepackage{algorithm}
\usepackage{algorithmicx}
\usepackage[]{algpseudocode}

\usepackage{enumitem}

\usepackage[labelfont=bf,textfont=it]{caption}
\usepackage{graphicx}
\usepackage{subcaption}
\usepackage{marginnote}

\usepackage{xcolor}

\definecolor{Base02}{HTML}{1E3668}

\usepackage{hyperref}
\usepackage[nameinlink,capitalize,noabbrev]{cleveref}
\crefformat{equation}{(#2#1#3)}
\crefrangeformat{equation}{(#3#1#4) to~(#5#2#6)}
\crefmultiformat{equation}{(#2#1#3)}%
{ and~(#2#1#3)}{, (#2#1#3)}{ and~(#2#1#3)}

\crefformat{cond}{#2condition #1#3}
\crefmultiformat{cond}{#2conditions #1#3}%
{ and~#2#1#3}{, #2#1#3}{ and~#2#1#3}

\Crefname{lemma}{Lemma}{Lemmas}
\Crefname{assumption}{Assumption}{Assumptions}

\hypersetup{
    colorlinks,
    linkcolor=Base02,
    citecolor=Base02,
    urlcolor=Base02
}

\renewcommand{\vec}{\mathbf}
\newcommand{\F}{\textsf{\textup{F}}}
\newcommand{\T}{\mathsf{T}}

\newcommand{\KS}{d_{\textup{KS}}}

\renewcommand{\d}{\mathrm{d}}
\usepackage{dsfont}
\newcommand{\bOne}{\mathds{1}}

\newcommand{\lmin}{\lambda_{\textup{min}}}
\newcommand{\lmax}{\lambda_{\textup{max}}}

\newcommand{\vt}[2][]{\hyperref[eqn:short_vector_def]{\color{black}\vec{t}_{#2}^{#1}}}
\newcommand{\vtol}[2][]{\hyperref[eqn:short_vector_def]{\color{black}\overline{\vec{t}}{}_{#2}^{#1}}}
\newcommand{\vd}[2][]{\hyperref[eqn:short_vector_def]{\color{black}\vec{d}_{#2}^{#1}}}
\newcommand{\vr}[2][]{\hyperref[eqn:short_vector_def]{\color{black}\vec{r}_{#2}^{#1}}}

\newcommand{\tta}{a}
\newcommand{\ttb}{b}

\newcommand{\mom}[1]{\hyperref[eqn:mom]{\color{black}\mathfrak{m}_{#1}}}
\newcommand{\momDel}[1]{\hyperref[eqn:mom_del]{\color{black}\mathfrak{m}_{#1}^{\Delta}}}
\newcommand{\opMax}[1]{\hyperref[eqn:op_max]{\color{black}M_{#1}}}
\newcommand{\muN}{\hyperref[eqn:VESD]{\color{black}\mu_N}}
\newcommand{\muVESD}{\hyperref[eqn:VESD]{\color{black}\mu_{\textup{VESD}}}}
\newcommand{\muk}{\hyperref[eqn:mu_N^k]{\color{black}\mu_k}}
\newcommand{\mukfp}{\hyperref[eqn:fp_mu_N^k]{\color{black}\overline{\mu}_k}}

\newcommand{\dirac}{\delta}

\DeclareMathAlphabet{\mathsfit}{T1}{\sfdefault}{\mddefault}{\sldefault}
\SetMathAlphabet{\mathsfit}{bold}{T1}{\sfdefault}{\bfdefault}{\sldefault}

\newcommand{\poly}{}
\newcommand{\ppi}{{\pi}}

\newcommand{\op}[1]{\hyperref[eqn:op]{\color{black}\poly{p}_{#1}}}
\newcommand{\opp}[1]{\hyperref[eqn:op]{\color{black}\poly{p}_{#1}'}}
\newcommand{\pT}{{\poly{T}}}
\newcommand{\pU}{{\poly{U}}}

\newcommand{\ii}{\ensuremath{\mathrm{i}}}
\DeclareMathOperator{\imag}{Im}
\DeclareMathOperator{\real}{Re}

\newcommand{\Tk}{\vec{T}_{k}}
\newcommand{\Qk}{\vec{Q}_{k}}
\newcommand{\Fk}{\hyperref[eqn:lanczos_factorization_fp]{\color{black}\vec{F}_{k}}}
\newcommand{\Hk}{\hyperref[eqn:lanczos_R_factorization_fp]{\color{black}\vec{H}_{k}}}
\newcommand{\Rk}{\hyperref[eqn:RD]{\color{black}\vec{R}_{k}}}
\newcommand{\Dk}{\hyperref[eqn:RD]{\color{black}\vec{D}_{k}}}

\newcommand{\Tkfp}{\overline{\vec{T}}_{k}}
\newcommand{\Qkfp}{\overline{\vec{Q}}{}_{k}}

\newcommand{\epslan}{\hyperref[def:lanczos_precision]{\color{black}\epsilon_{\textup{lan}}}}

\newcommand{\Fkp}{\hat{\vec{F}}_k}
\newcommand{\Hkp}{\hat{\vec{H}}_k}
\newcommand{\Tkfpp}{\hat{\overline{\vec{T}}}_{k}}
\newcommand{\epslanp}{\hat{\epsilon}_{\textup{lan}}}
\newcommand{\muNp}{\hat{\mu}_N}
\newcommand{\mukfpp}{\hat{\overline{\mu}}_k}

\title{Stability of the Lanczos algorithm on matrices with regular spectral distributions}

\author{Tyler Chen\thanks{New York University, \href{mailto:tyler.chen@nyu.edu}{\texttt{tyler.chen@nyu.edu}}}
\and Thomas Trogdon\thanks{University of Washington, \href{mailto:trogdon@uw.edu}{\texttt{trogdon@uw.edu}}}}
\date{}

\newcommand\blfootnote[1]{%
  \begingroup
  \renewcommand\thefootnote{}\footnote{#1}%
  \addtocounter{footnote}{-1}%
  \endgroup
}

\begin{document}
\maketitle
\blfootnote{\textbf{Funding.} This material is based on work supported by the National Science Foundation under Grant Nos. DGE-1762114 (TC), DMS-1945652 (TT).
Any opinions, findings, and conclusions or recommendations expressed in this material are those of the authors and do not necessarily reflect the views of the National Science Foundation.}
\begin{abstract}
We study the stability of the Lanczos algorithm run on problems whose eigenvector empirical spectral distribution is near to a reference measure with well-behaved orthogonal polynomials.
We give a backwards stability result which can be upgraded to a forward stability result when the reference measure has a density supported on a single interval with square root behavior at the endpoints.
Our analysis implies the Lanczos algorithm run on many large random matrix models is in fact forward stable, and hence nearly deterministic, even when computations are carried out in finite precision arithmetic.
Since the Lanczos algorithm is not forward stable in general, this provides yet another example of the fact that random matrices are far from ``any old matrix'', and care must be taken when using them to test numerical algorithms.
\end{abstract}

\section{Introduction}

The Lanczos algorithm is unstable in the sense that, even on the simplest problems, the output of the algorithm in finite precision arithmetic may be very different than what would have been obtained in exact arithmetic.
Despite this, the Lanczos algorithm is among the most important algorithms in numerical linear algebra and is commonly used for a wide variety of fundamental linear-algebraic tasks including approximating eigenvalues and eigenvectors, the product of a matrix function on a vector, and quadratic forms involving matrix functions.
Understanding the behavior of the Lanczos algorithm in finite precision arithmetic has been of interest since the introduction of the algorithm some 70 years ago \cite{lanczos_50,golub_oleary_89,parlett_98,meurant_06,carson_liesen_strakos_22}.

\begin{algorithm}
\caption{Lanczos algorithm}
\label{alg:lanczos}
\begin{algorithmic}[1]
    \Procedure{Lanczos}{$\vec{A}, \vec{b}, k$}
    \State \( \vec{q}_0 = \vec{b} / \|\vec{b}\| \), \( \beta_{-1} = 0 \), \( \vec{q}_{-1} = \vec{0} \)
    \For {\( n=0,1,\ldots, k-1 \)}
        \State \( \tilde{\vec{q}}_{n+1} = \vec{A} \vec{q}_{n} - \beta_{n-1} \vec{q}_{n-1} \)
        \State \( \alpha_{n} = \tilde{\vec{q}}_{n+1}^\T \vec{q}_n \)
        \State \( \hat{\vec{q}}_{n+1} = \tilde{\vec{q}}_{n+1} - \alpha_{n} \vec{q}_i \)
        \State \( \beta_{n} = \| \hat{\vec{q}}_{n+1} \| \)
        \State \( \vec{q}_{n+1} = \hat{\vec{q}}_{n+1} / \beta_{n} \)
    \EndFor
\EndProcedure
\end{algorithmic}
\end{algorithm}

Throughout, $\vec{A}$ will be an $N\times N$ real symmetric matrix and $\vec{b}$ a unit-norm vector of length $N$.
From $(\vec{A},\vec{b})$ we obtain the eigenvector empirical spectral distribution (VESD) defined by
\begin{equation}
    \label{eqn:VESD}
    \muN(\d{x}) = \muVESD(\d{x}; \vec A, \vec{b}):= \sum_{n=1}^{N} (\vec{b}^\T \vec{u}_n )^2 \:\dirac_{\lambda_n}(\d x),
\end{equation}
where $(\lambda_n, \vec{u}_n)$ are the eigenvalue-vector pairs of $\vec{A}$ and $\dirac_c$ is the Dirac delta distribution centered at $c$.
We use the former notation when $\vec A$ and $\vec{b}$ are clear from context.
When run on $(\vec{A}, \vec{b})$ for $k$ iterations \emph{in exact arithmetic}, the Lanczos algorithm (\cref{alg:lanczos}) outputs an orthonormal basis $[\vec{q}_0, \ldots, \vec{q}_k]$ for the Krylov subspace
\begin{equation*}
    \operatorname{span}\{\vec{b}, \vec{A}\vec{b}, \ldots, \vec{A}^{k} \vec{b} \}
\end{equation*}
and coefficients $(\alpha_0, \ldots, \alpha_{k-1})$, $(\beta_0, \ldots, \beta_{k-1})$ for a three-term recurrence satisfied by the basis vectors.
In matrix form, this recurrence can be written
\begin{equation}
\label{eqn:lanczos_three_term}
    \vec{A} {\Qk} = {\Qk} {\Tk} + {\beta}_{k-1} {\vec{q}}_{k} \vec{e}_{k-1}^\T,
\end{equation}
where $\vec{e}_{k-1} = [0,\ldots, 0,1]^{\T}$ and
\begin{equation}
    \label{eqn:QT}
    \Qk
    = \begin{bmatrix}
    |&|&&|\\
    \vec{q}_0 & \vec{q}_1 & \cdots & \vec{q}_{k-1}\\
    |&|&&|
    \end{bmatrix}
    ,\quad
    \Tk = \operatorname{tridiag}
    \left(\hspace{-.75em} \begin{array}{c}
        \begin{array}{cccc} \beta_0 & \beta_1 & \cdots & \beta_{k-2} \end{array} \\
        \begin{array}{ccccc} \alpha_0 & \alpha_1 & \cdots& \cdots & \alpha_{k-1} \end{array} \\
        \begin{array}{cccc} \beta_0 & \beta_1 & \cdots & \beta_{k-2} \end{array} 
    \end{array} \hspace{-.75em}\right).
\end{equation}

The Lanczos algorithm run on $(\vec{A},\vec{b})$ is mathematically equivalent to the Stieltjes procedure for computing the recurrence coefficients for the orthogonal polynomials of the VESD $\muN$ \cite{gautschi_04}.
It is common to refer to the matrix $\Tk$ as the Jacobi matrix associated with $\muN$, and from this point on, we will make no distinction between the Lanczos algorithm in exact arithmetic and Stieltjes procedure. 
The $k$-point Gaussian quadrature rule for $\muN$ will be written as $\muk$, and is equal to the VESD for $(\Tk, \vec{e}_0)$, where $\vec{e}_0 = [1,0,\ldots, 0]^\T$.
That is, 
\begin{equation}
    \label{eqn:mu_N^k}
    {\mu}_k(\d{x}) = \muVESD(\d{x}; \vec{T}_k, \vec{e}_0) := \sum_{n=1}^{k} ( \vec{e}_0^\T {\vec{s}}_{n} )^2 \:\dirac_{\theta_{n}}(\d x),
\end{equation}
where $(\theta_{n}, \vec{s}_{n})$ are the eigenvalue-vector pairs of $\Tk$.
Note that \cref{eqn:VESD,eqn:mu_N^k} coincide once $k$ is large enough that the dimension of the Krylov subspace stops growing. 
This occurs once $k$ is equal to the number of points of support for $\muN$.
However, implicit in our analysis, is the assumption $k\ll N$.

When the Lanczos algorithm is run on $(\vec{A},\vec{b})$ for $k$ iterations in \emph{finite precision arithmetic}, the vectors $[\overline{\vec{q}}_0, \ldots, \overline{\vec{q}}_k]$ and coefficients $(\overline{\alpha}_0, \ldots, \overline{\alpha}_{k-1})$, $(\overline{\beta}_0, \ldots, \overline{\beta}_{k-1})$ generated by the algorithm may be nothing like their exact arithmetic counterparts. 
Analogously to \cref{eqn:mu_N^k}, we define the VESD $\mukfp$ for  $(\Tkfp, \vec{e}_0)$ by
\begin{equation}
    \label{eqn:fp_mu_N^k}
    \mukfp(\d{x}) = \muVESD(\d{x}; \overline{\vec{T}}_k, \vec{e}_0):= \sum_{n=1}^{k} ( \vec{e}_0^\T \overline{\vec{s}}_{n} )^2 \:\dirac_{\overline{\theta}_{n}}(\d x),
\end{equation}
where $(\overline{\theta}_{n}, \overline{\vec{s}}_{n})$, $n=1, \ldots, k$ are the eigenvalues-vectors pairs of $\Tkfp$, the symmetric tridiagonal matrix with diagonal $(\overline{\alpha}_0, \ldots, \overline{\alpha}_{k-1})$ and sub/super-diagonals $(\overline{\beta}_0, \ldots, \overline{\beta}_{k-2})$.

In numerical analysis, there are a number of notions of stability. 
Arguably, the most common are forward stability and backward stability, which we now describe in the context of the Lanczos algorithm.
\begin{definition}
The Lanczos algorithm run for $k$ iterations in finite precision arithmetic on an input $(\vec{A},\vec{b})$ to obtain output $\Tkfp$ is
\begin{itemize}[topsep=.15em,itemsep=.15em]
    \item \emph{forward stable} if $\Tkfp$ is near $\Tk$, the output of exact Lanczos run on $(\vec{A},\vec{b})$, and
    \item \emph{backward stable} if $\Tkfp$ is the Jacobi matrix for a nearby input $(\vec{A}_*,\vec{b}_*)$; that is, if exact Lanczos run on $(\vec{A}_*,\vec{b}_*)$ produces $\Tkfp$.
\end{itemize}
\end{definition}
For the purposes of this paper, \emph{nearby} is understood to mean differing by an amount with a polynomial dependence on $k$ and the machine precision $\epsilon_{\textup{mach}}$ (in some reasonable metric).
Ideally, the dependence on $\epsilon_{\textup{mach}}$ is linear, and when $\epsilon_{\textup{mach}} = 0$, the exact arithmetic behavior is recovered.
As with most stability analyses of the Lanczos algorithm, the value of our work is not in the numerical value of the bounds themselves, but rather in the intuition the bounds convey.
For instance, situations in which our bounds depend exponentially on $k$ provide insight into problems on which the Lanczos algorithm is potentially unstable.
In line with this philosophy, we will not attempt to optimize polynomial dependencies in $k$; instead, we aim to minimize the complexity of the statements and proofs of our results.

As noted above, understanding the stability of the Lanczos algorithm in finite precision arithmetic has been an active area of the research for the past half century.
Perhaps the most well-known work is that of Paige \cite{paige_71,paige_76,paige_80} (which we discuss further in \cref{sec:paige}) and Greenbaum \cite{greenbaum_89}.
In addition, a number of books and notes contain extensive writing on the topic \cite{parlett_98,meurant_06}.

Greenbaum's analysis, which is the preeminent backwards stability analysis of the Lanczos algorithm, proves the existence of a nearby problem $(\vec{A}_*,\vec{b}_*)$ such that, when Lanczos is run on $(\vec{A}_*,\vec{b}_*)$ for $k$ iterations in exact arithmetic, $\Tkfp$ is output.
Here nearby roughly means (i) every eigenvalue of $\vec{A}_*$ is near an eigenvalue of $\vec{A}$, and (ii) $\muVESD(\cdot\,; \vec{A}_*, \vec{b}_*)$ 
is near to $\muN = \muVESD(\cdot\,; \vec A, \vec{b})$.
This result is very strong in that it applies to any input $(\vec{A},\vec{b})$.
The main drawbacks are that the nearby problem $(\vec{A}_*,\vec{b}_*)$ is of a different dimension than the original problem, and the precise definition of nearby has a sub-linear dependence on the machine precision which is generally believed to be pessimistic. 
In addition, the proofs of the result are quite technical.

Another important stability result, which seems to have been mostly overlooked by the numerical analysis community, is  Knizhnerman's analysis of the modified Chebyshev moments of $\mukfp$ \cite{knizhnerman_96}.
In particular, Knizhnerman shows that the modified Chebyshev moments of $\mukfp$ are near those of $\muk$.
This paper extends Knizhnerman's work.

\subsection{Motivation}
Testing numerical algorithms on random matrices is a widespread practice.
However, as noted by Edelman and Rao \cite{edelman_rao_05}, 
\begin{quote}
It is a mistake to link psychologically a random matrix with the intuitive notion of a `typical' matrix or the vague concept of `any old matrix'.
\end{quote}
In particular, numerical algorithms run on random matrices may fail to capture the typical behavior of the algorithm on an arbitrary matrix.
The Lanczos algorithm is a clear example of this.
While the algorithm is not forward stable in general, when run on a large random matrix, drawn from a suitable distribution, $\Tkfp$ matches closely to $\Tk$, at least while number of iterations $k$ is sufficiently small compared to the dimension $N$. 
We illustrate this phenomenon numerically in \cref{fig:motivating_experiment}.

\begin{figure}[htb]
    \centering
    \begin{subfigure}[t]{0.48\textwidth}
         \centering
         \includegraphics[scale=.7]{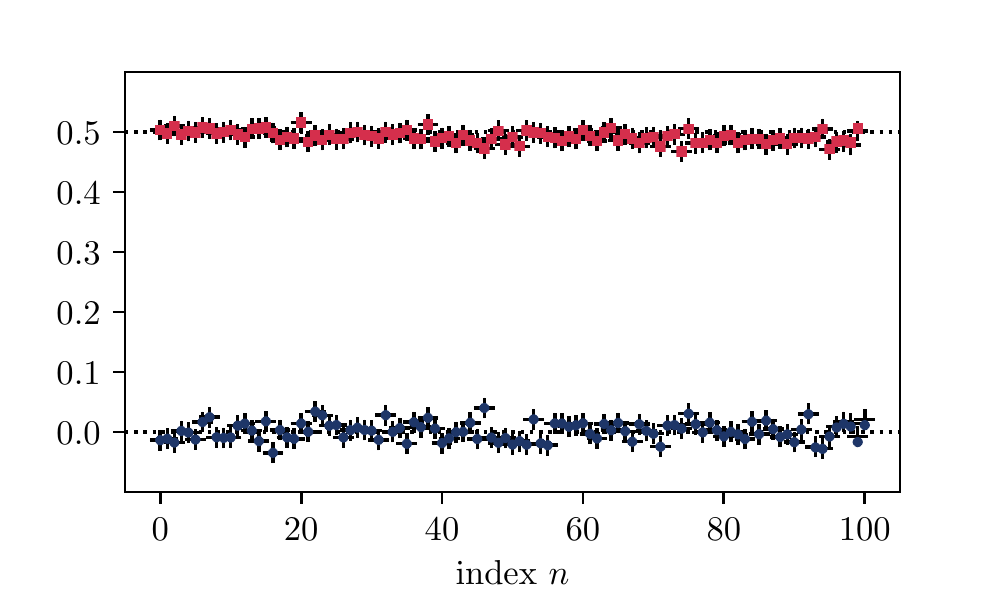}
         \caption{Recurrence coefficients $\overline{\alpha}_n$ \textup{(\includegraphics[scale=.8]{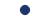})} and $\overline{\beta}_n$ \textup{(\includegraphics[scale=.8]{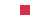})}. Exact arithmetic counterparts shown as pluses \textup{(\includegraphics[scale=.8]{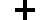})} and limiting values shown as dotted lines \textup{(\includegraphics[scale=.8]{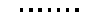})}.}
         \label{fig:motivating_experiment_coeff}
    \end{subfigure}
    \hfill
    \begin{subfigure}[t]{0.48\textwidth}
         \centering
         \includegraphics[scale=.7]{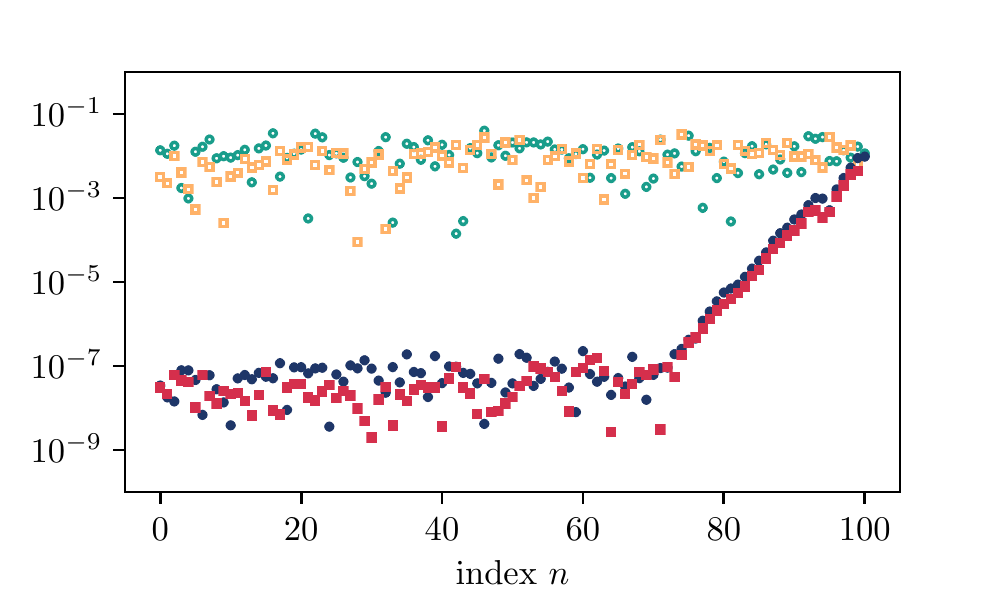}
         \caption{Forward error of recurrence coefficients $|\alpha_n - \overline{\alpha}_n|$ \textup{(\includegraphics[scale=.8]{imgs/legend/circ.pdf})} and $|\beta_n - \overline{\beta}_n|$ \textup{(\includegraphics[scale=.8]{imgs/legend/square_red.pdf})}
         and distance to limiting values $|0 - \overline{\alpha}_n|$ \textup{(\includegraphics[scale=.8]{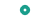})} and $|1/2 - \overline{\beta}_n|$ \textup{(\includegraphics[scale=.8]{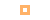})}.}
         \label{fig:motivating_experiment_coeff_fe}
    \end{subfigure}
    \caption{Here $(\vec{A},\vec{b})$ corresponds to a $2000\times 2000$ random matrix, drawn from the Gaussian orthogonal ensemble (see \cref{sec:wigner}), and independent vector.
    In the large $N$ limit, the VESD of matrices drawn from this ensemble converge to the semicircle distribution on $[-1,1]$ (density $\propto \sqrt{1-x^2}$).
    Therefore the Lanczos coefficients $\alpha_i$ and $\beta_i$ from the ``exact'' computation (with reorthogonalization in quadruple precision arithmetic) respectively converge to $1/2$ and $0$; i.e. the Lanczos algorithm exhibits deterministic behavior.
    In our particular experiment we observe fluctuations on the order of $10^{-2}$ around the limiting values due to finite $N$ effects.
    Remarkably, the coefficients $\overline{\alpha}_n$ and $\overline{\beta}_n$ output by the Lanczos algorithm run in \emph{single precision} floating point arithmetic without reorthogonalization are within the unit roundoff ($\approx 10^{-7}$) of $\alpha_n$ and $\beta_n$, at least while $n$ is sufficiently small; i.e. the algorithm is forward stable.
}
    \label{fig:motivating_experiment}
\end{figure}

The aim of this paper is to provide an intuitive explanation for the observation that the Lanczos algorithm is stable on problems whose VESD are sufficiently regular.
More specifically, our approach extends the work of Knizhnerman \cite{knizhnerman_96} to prove the existence of a measure $\mu_*$ near to $\muN$ whose moments agree with $\mukfp$ through degree $2k-1$, at least when the VESD of $(\vec{A},\vec{b})$ is sufficiently regular.
In fact, under certain regularity conditions, we show there exists a vector $\vec{b}_*$ near to $\vec{b}$ such that Lanczos run on $(\vec{A},\vec{b}_*)$ in exact arithmetic for $k$ iterations outputs $\Tkfp$.
In other words, on a restricted set of inputs, we provide a simpler proof for a stronger version of Greenbaum's results.
We then provide forward stability results by analyzing the orthogonal polynomials of slightly perturbed measures.
This shows that, on many large random matrix models, the output of the Lanczos algorithm is nearly deterministic, even when computations are carried out in finite precision arithmetic.
Our analysis is accompanied by numerical experiments and several explicit examples.

\subsection{Notation}

Throughout this work, we use $\Lambda(\vec A)$ to refer to the spectrum of a matrix.   For a function $f: U \to \mathbb C$ with $S \subseteq U$, we define $\|f\|_S := \sup_{x \in S} |f(x)|$.  For a vector $\vec{b}$, $\|\vec{b}\|$ refers to the Euclidean 2-norm and $\|\vec A\|$ gives the associated induced operator norm for a matrix $\vec A$.
The $n$-th canonical basis vector, indexed from 0, is $\vec{e}_n$.
The Kolmogorov--Smirnov distance between two measures $\nu_1$ and $\nu_2$ is $\KS(\nu_1,\nu_2) := \sup_{x\in\mathbb{R}} |\nu_1((-\infty,x]) - \nu_2((-\infty,x])$.  All measures we consider will be Borel measures.  Indeed, all measures will be either fully discrete or have a continuous density.

\section{Setup and background}
\label{sec:setup}

Let $\mu$ be a unit-mass measure with support contained in $[\tta,\ttb]$.
We will refer to $\mu$ as the \emph{reference measure}, and it will be helpful to think of $\mu$ as near to $\muN$; for instance $\mu = \muN$ or $\mu$ being the limiting measure for the VESD of a large random matrix ensemble.
In particular, we will typically have $[\tta,\ttb]\approx [\lmin(\vec{A}),\lmax(\vec{A})]$. 
We denote by $\op{n} = \op{n}(\cdot\,;\mu)$, $n\geq 0$ the orthonormal polynomials for $\mu$. \label{eqn:op}
That is, the $\op{n}(\cdot\,;\mu)$ satisfy\footnote{These polynomials are constructed by performing Gram--Schmidt on the monomial basis in order of increasing degree and are normalized to have a positive leading coefficient.}
\[ 
\int \op{n}(x;\mu) \op{i}(x;\mu) \mu(\d{x}) = \bOne(n=i),
\]
where $\bOne(\texttt{true}) = 1$ and $\bOne(\texttt{false}) = 0$.

The modified moments of a measure $\nu$ with respect to the orthogonal polynomials of $\mu$ are defined by
\begin{equation}
\label{eqn:mom}
    \mom{n}(\nu;\mu) := \int \op{n}(x;\mu) \nu(\d x)
    ,\qquad
    n\geq 0.
\end{equation}
Clearly $\mom{n}(\mu;\mu) = \bOne(n=0)$ and $\mom{0}(\mu;\mu) = \mom{0}(\muN;\mu) = \mom{0}(\muk;\mu) = \mom{0}(\mukfp;\mu) = 1$.

As mentioned in the introduction, \cite{knizhnerman_96} shows that the modified moments of $\muN$ and $\mukfp$ through degree $2k-1$ are close when $\mu$ is a properly scaled and shifted version of the orthogonality measure for the Chebyshev polynomials of the first kind.
A similar statement, with some polynomial losses in $k$, can therefore be expected to hold for any $\mu$ whose orthogonal polynomials have a Chebyshev series representation with reasonable coefficients.

The idea underlying our analysis is to construct a (potentially signed) measure $\mu_*$ as a perturbation to the reference measure $\mu$:
\begin{equation}
\label{eqn:new_measure}
    \mu_*(\d{x}) :=  \left( 1 + \poly{h}(x) \right)\mu(\d{x})
    ,\quad 
    \poly{h}(x) :=\sum_{n=0}^{2k-1} \big(\mom{n}(\mukfp;\mu) - \mom{n}(\mu;\mu) \big) \op{n}(x;\mu).
\end{equation}
This construction ensures $\mu_*$ has the same moments as $\mukfp$ through degree $2k-1$ and the same moments as $\mu$ for higher degrees.
Indeed, by definition, the $\op{n}$ are orthonormal with respect to $\mu$, so 
\begin{align*}
    \mom{n}(\mu_*;\mu) &= \int \op{n}(x;\mu) \mu_*(\d{x})
    \\&= \int \op{n}(x;\mu) \mu(\d{x})
    + \int \op{n}(x;\mu)  \sum_{i=0}^{2k-1} \big(\mom{n}(\mukfp;\mu) - \mom{n}(\mu;\mu) \big) \op{i}(x;\mu) \mu (\d{x})
    \\&= \mom{n}(\mu;\mu) + \sum_{i=0}^{2k-1} \big(\mom{n}(\mukfp;\mu) - \mom{n}(\mu;\mu) \big)
    \int \op{n}(x;\mu) \op{i}(x;\mu) \mu (\d{x}) 
    \\&= \mom{n}(\mu;\mu) + \sum_{i=0}^{2k-1} \big(\mom{n}(\mukfp;\mu) - \mom{i}(\mu;\mu) \big) \bOne(i = n)
    \\&= 
    \begin{cases}
    \mom{n}(\mukfp;\mu) & i =0,1,\ldots, 2k-1 \\
    \mom{n}(\mu;\mu) & i=2k, 2k+1, \ldots .
    \end{cases}
\end{align*}
Since the moments of $\mu_*$ match those of $\mukfp$ through degree $2k-1$, when the Stieltjes procedure is run on $\mu_*$ for $k$ iterations, $\Tkfp$ is the output.

Introduce the quantities
\begin{align}
\opMax{k}(\mu; [\tta,\ttb]) &:= \max_{n\leq 2k-1} \|\op{n}(\cdot\,;\mu)\|_{[\tta,\ttb]}\label{eqn:op_max}
,\\
\momDel{k}(\nu_1,\nu_2;\mu) &:= \max_{n\leq 2k-1} |\mom{n}(\nu_1;\mu) - \mom{n}(\nu_2;\mu) |.\label{eqn:mom_del}
\end{align}
Clearly 
\begin{equation}
\label{eqn:h_bd}
    \|\poly{h}\|_{[\tta,\ttb]} 
    \leq \sum_{i=0}^{2k-1} |\mom{n}(\mukfp;\mu) - \mom{n}(\mu;\mu)| \opMax{k}(\mu; [\tta,\ttb])
    \leq 2k \momDel{k}(\mukfp,\mu;\mu) \opMax{k}(\mu; [\tta,\ttb]),
\end{equation}
so if $\momDel{k}(\mukfp,\mu;\mu)$ is sufficiently small relative to the reciprocal of $k \opMax{k}(\mu; [\tta,\ttb])$, then $\|h\|_{[\tta,\ttb]} \leq 1$ and $\mu_*$ is a well-defined non-negative measure.
In this case, if $\mu\approx \muN$ then we also have $\mu_* \approx \muN$; i.e. backwards stability.

\begin{remark}\label[remark]{rem:backwards_finite}
If we take $\mu = \muN$, then, assuming $\|\poly{h}\|_{\Lambda(\vec{A})} \leq 1$, $\mu_*$ is the VESD of $(\vec{A},\vec{b}_*)$, where 
\begin{equation*}
    \vec{b}_* =  (\vec{I} + \poly{h}(\vec{A}))^{1/2} \vec{b}.
\end{equation*}
This is a perturbation of $\vec{b}$ in the sense that 
\begin{equation*}
    \|\vec{b} - \vec{b}_*\| \leq \| \vec{I} - (\vec{I} + \poly{h}(\vec{A}))^{1/2} \| \| \vec{b}\|
    \leq \| \poly{h}(\vec{A}) \|.
\end{equation*}
\end{remark}
In \cref{fig:GOE_backwards} we illustrate this approach for $\mu = \muN$, where $\vec{A}$ is the same random matrix as used in \cref{fig:motivating_experiment}.
Bounds for $\mu_*$ are derived in \cref{sec:backwards,sec:forwards}.

\begin{figure}[ht]
    \centering
    \begin{subfigure}[t]{0.48\textwidth}
         \centering
         \includegraphics[scale=.7]{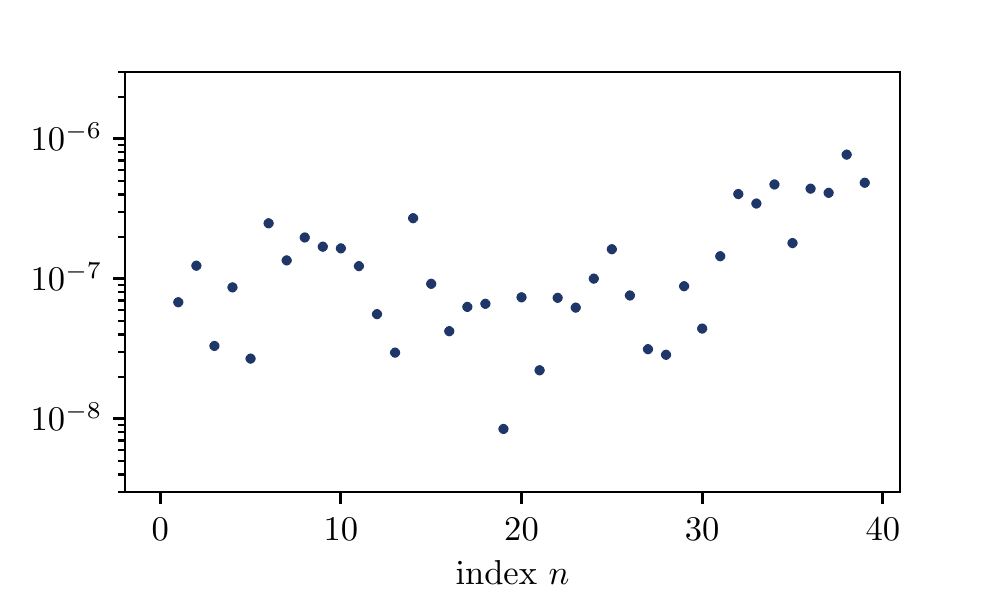}
         \caption{Forward error of modified moments $|\mom{n}(\muN;\muN) - \mom{n}(\mukfp;\muN)|$ \textup{(\includegraphics[scale=.8]{imgs/legend/circ.pdf})}.
         Note that $\mom{n}(\muN;\muN) = \mom{n}(\muk;\muN)$ for $n\leq 2k-1$.}
         \label{fig:GOE_backwards_moments} 
    \end{subfigure}
    \hfill
    \begin{subfigure}[t]{0.48\textwidth}
         \centering
         \includegraphics[scale=.7]{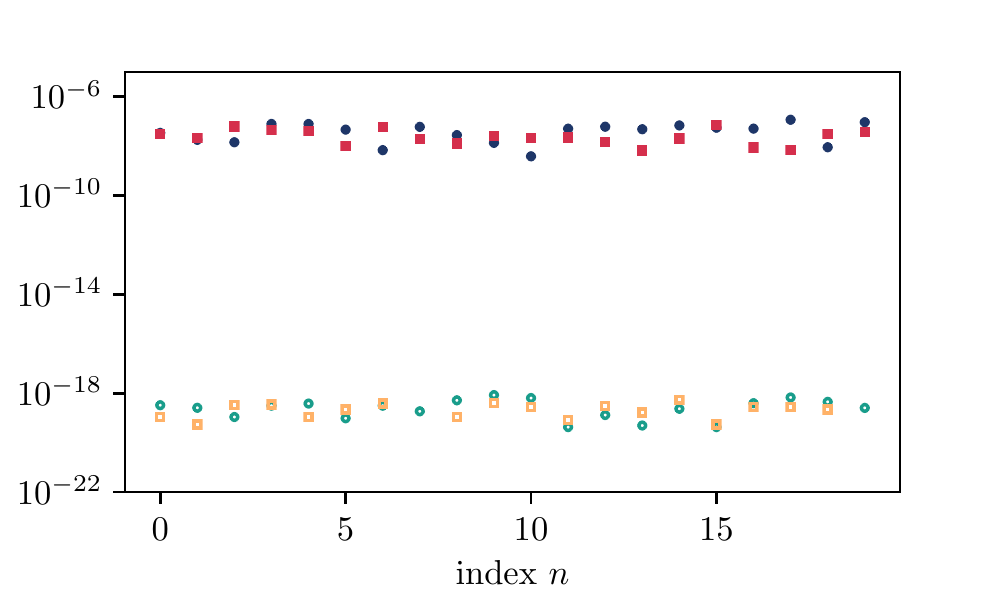}
         \caption{Forward error of recurrence coefficients $|\alpha_n - \overline{\alpha}_n|$ \textup{(\includegraphics[scale=.8]{imgs/legend/circ.pdf})} and $|\beta_n - \overline{\beta}_n|$ \textup{(\includegraphics[scale=.8]{imgs/legend/square_red.pdf})} and $|\alpha_n^* - \overline{\alpha}_n|$ \textup{(\includegraphics[scale=.8]{imgs/legend/circ_green.pdf})} and $|\beta_n^* - \overline{\beta}_n|$ \textup{(\includegraphics[scale=.8]{imgs/legend/square_yellow.pdf})}, where $\alpha_n^*,\beta_n^*$ correspond to an ``exact'' computation with $(\vec{A}_*,\vec{b}_*)$.}
         \label{fig:GOE_backwards_coeff_fe}
    \end{subfigure}
    \begin{subfigure}[b]{1\textwidth}
         \centering
         \includegraphics[scale=.7]{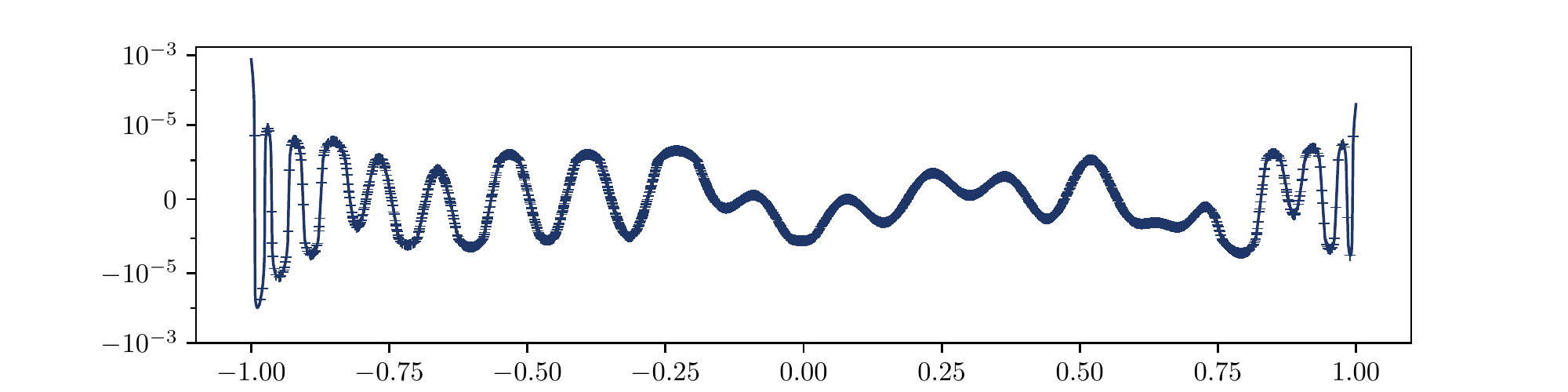}
         \caption{Perturbation function $\poly{h}(x)$ \textup{(\includegraphics[scale=.8]{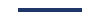})} and value at eigenvalues of $\vec{A}$ \textup{(\includegraphics[scale=.8]{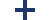})}.}
         \label{fig:GOE_backwards_h}
    \end{subfigure}
    \caption{We use the same $2000\times 2000$ random matrix $\vec{A}$ and fixed vector $\vec{b}$ from \cref{fig:motivating_experiment}.
    After running the Lanczos algorithm on $(\vec{A},\vec{b})$ in single precision finite precision arithmetic without reorganization, we use \cref{eqn:new_measure} and \cref{rem:backwards_finite} to construct a slightly perturbed $\vec{b}_*$. 
    Lanczos run on $(\vec{A},\vec{b}^*)$ ``exactly'' (with reorthogonization in quadruple precision floating point arithmetic) produces a nearly identical output as the original single precision finite precision computation.
    For reference, $\|\vec{b} - \vec{b}_*\|\approx 3.6\cdot 10^{-5}$ is only a few orders of magnitude above the machine precision in which the original computation was carried out.
    }
    \label{fig:GOE_backwards}
\end{figure}

There is a fundamental equivalence between the Jacobi matrix $\Tk$ produced by the Stieltjes procedure and the modified moments of $\muN$ through degree $2k-1$ with respect to some fixed measure.\footnote{These quantities are also equivalent to the $k$-point Gaussian quadrature rule $\muk$ for $\muN$.}
The conditioning of the map from moments to Jacobi matrix is very poor in general \cite{gautschi_82,fischer_96,gautschi_04,oleary_strakos_tichy_07}, so even if the modified moments of $\muN$ and $\mukfp$ are close, this does not generally imply the corresponding Jacobi matrices $\Tk$ and $\Tkfp$ are close.
However, in certain situations when $\muN$ is sufficiently regular, then the conditioning of the map from modified moments to Jacobi matrices is well-conditioned \cite{fischer_96} and we can expect the Lanczos algorithm to be forward stable.
Bounds for $\Tkfp$ are derived in \cref{sec:forwards}.

\subsection{Perturbed Lanczos recurrences}
\label{sec:paige}

To carry out our analysis, we require some understanding of the behavior of the Lanczos methods in finite precision arithmetic.
Much is known about this topic \cite{parlett_98,meurant_06}, but we summarize only what is needed for our analysis.

The finite precision arithmetic outputs $\Qkfp$ and $\Tkfp$ no longer satisfy the three-term Lanczos recurrence \cref{eqn:lanczos_three_term} exactly. 
Instead, they satisfy a perturbed recurrence
\begin{equation}
    \label{eqn:lanczos_factorization_fp}
    \vec{A} \Qkfp = \Qkfp \Tkfp + \overline{\beta}_{k-1} \overline{\vec{q}}_{k} \vec{e}_{k-1}^\T + \Fk,
\end{equation}
where the perturbation term $\Fk$ accounts for local rounding errors made by the algorithm.
Since $\vec{F}_{k}$ does not involve accumulated errors, but rather errors made over a single iteration of the Lanczos algorithm, it can intuitively be expected to be small.

We denote by \( \Rk \) and $\Dk$ the strictly upper triangular and diagonal parts of \( \Qkfp^\T\Qkfp \) so that $\Qkfp^\T\Qkfp = \Rk + \Rk^\T + \Dk$; i.e. 
\begin{equation}
    \label{eqn:RD}
    \Rk := \operatorname{strict-triu}(\Qkfp^\T\Qkfp)
    ,\qquad
    \Dk := \operatorname{diag}(\Qkfp^\T\Qkfp).
\end{equation}
Since $\Dk-\vec{I}$ accounts for errors made when normalizing the Lanczos basis vectors, we expect it to be small.
On the other hand, since rounding errors accumulate in the columns of $\Qkfp$, the entries of $\Rk$ need not be small. 
In fact, in many situations $\Rk$ can have entries of size $O(1)$ indicating a complete loss of orthogonality in the Lanczos basis vectors.

The matrix $\Rk$ satisfies a perturbed three-term recurrence 
\begin{equation}
    \label{eqn:lanczos_R_factorization_fp}
    \Tkfp \Rk = \Rk\Tkfp  + \overline{\beta}_{k-1}\Qkfp^\T \overline{\vec{q}}_k\vec{e}_{k-1}^\T + \Hk,
\end{equation}
with an upper triangular perturbation term $\Hk$. 
Straightforward algebraic manipulations of \cref{eqn:lanczos_factorization_fp} using \cref{eqn:RD} show that $\Hk$ should be expected to be small as well.

Finally, we define $\eta_k\geq0$ to be the smallest value such that
\begin{equation}
    \label{eqn:eta}
    \Lambda(\Tkfp) \subseteq [ \lmin(\vec{A}) - \eta_k, \lmax(\vec{A}) + \eta_k].
\end{equation}

\begin{definition}\label[definition]{def:lanczos_precision}
We say the Lanczos algorithm was run for $k$ iterations with precision
 $\epslan$ if
\begin{equation*}
    \| \Fk \| \leq \| \vec{A} \| \epslan
    ,\qquad
     \| \Dk - \vec{I} \| \leq \epslan
    ,\qquad
    \| \Hk \| \leq \| \vec{A} \| \epslan
    ,\qquad
    \eta_k \leq \| \vec{A} \| \epslan
\end{equation*}
where $\Fk$, $\Dk$, $\Hk$, and $\eta_k$ are defined in \cref{eqn:lanczos_factorization_fp,eqn:RD,eqn:lanczos_R_factorization_fp,eqn:eta}.
\end{definition}

Bounds for $\|\vec{F}_k\|$, $\|\vec{D}_k\|$, and $\|\vec{H}_k\|$ appear in \cite{paige_71,paige_76} and the most well-known bound for $\eta_k$ appears in \cite{paige_80}. 
More recently, Paige has shown a bound for $\eta_k$ \cite[Theorem A.1]{paige_19} which, when combined with \cite[Theorem 3.1]{paige_10} improves the dependence on $k$ in the bound for $\eta_k$ over \cite{paige_80}.

\begin{proposition}[informal; see \cite{paige_70,paige_80}]\label[proposition]{thm:paige}
When the Lanczos algorithm is run for $k$ iterations on a computer with relative machine precision $\epsilon_{\textup{mach}} < O(1/k)$, then 
\begin{equation*}
\epslan = \max\left\{ N , \frac{\| |\vec{A}|_{\textup{entry}} \|}{ \| \vec{A} \|} \operatorname{row-nnz}(\vec{A}) \right\} O(\operatorname{poly}(k)\epsilon_{\textup{mach}}).
\end{equation*}
Here $\operatorname{row-nnz}(\vec{A})$ is the largest number of nonzero entries in a row of $\vec{A}$ and $|\vec{A}|_{\textup{entry}}$ is the entry-wise absolute value of $\vec{A}$; i.e. $[|\vec{A}|_{\textup{entry}}]_{i,j} = |[\vec{A}]_{i,j}|$.
\end{proposition}
Paige's analysis is far more precise than \cref{thm:paige}.
In particular the analyses result in explicit bounds on the powers of $k$ and the constants in front of each of the quantities in \cref{def:lanczos_precision}.
In Paige's analyses, terms of order $(\epsilon_{\textup{mach}})^2$ are typically discarded for clarity, but the results are essentially the same if the higher order terms are accounted for.

\begin{remark}
It always holds that $\||\vec{A}|_{\textup{entry}} \| \leq N^{1/2} \| \vec{A} \|$, so for uniformly sparse matrices with up to $O(N^{1/2})$ entries per row, the Lanczos algorithm is run with precision $\epslan$ if $\epsilon_{\textup{mach}} = O(\epslan/(\operatorname{poly}(k) N) )$.
\end{remark}

\section{Backwards stability}
\label{sec:backwards}

Our first main result shows that constructing $\mu_*$ as in \cref{eqn:new_measure} gives a nearby problem to $\mu_N$ when the reference measure $\mu$ is chosen suitably.
\begin{restatable}[Backwards stability]{theorem}{thmBackwards}\label{thm:backwards}
There exist absolute constants $C, D$ such that, for $(\vec{A},\vec{b})$ with VESD $\muN$ and any unit-mass measure $\mu$ with support contained in $[\tta,\ttb]$, the following statement holds:

Suppose Lanczos is run on $(\vec{A},\vec{b})$ for $k\geq 1$ iterations with precision $\epslan < 1/(\sigma Ck^2)$, where $\sigma :=\max\left\{ 1, {2\| \vec{A} \|}/{(\ttb-\tta)} \right\}$, to produce $\mukfp$ and
\[\operatorname{supp}(\muN) \subseteq [\tta-(\ttb-\tta)/(32k^2),\ttb+(\ttb-\tta)/(32k^2)].\]
Then the (possibly-signed) measure $\mu_*$ constructed in \cref{eqn:new_measure} is close to $\mu$ in the sense that
\begin{enumerate}[label=(\alph*),topsep=.15em,itemsep=.15em]
\item \label{thm:backwards:nearby_input}
$\momDel{k}(\mu_*,\muN;\mu) \leq  D \sigma \opMax{k}(\mu;[\tta,\ttb]) k^{3} \epslan$, and
\item \label{thm:backwards:h_bd}
$\|\poly{h}\|_{[\tta,\ttb]} \leq 2k\opMax{k}(\mu;[\tta,\ttb])    (\momDel{k}(\mu_*,\muN;\mu)+\momDel{k}(\muN,\mu;\mu))$
\end{enumerate}
Furthermore, provided that $\| h \|_{[\tta,\ttb]} < 1$, $\mu_*$ is non-negative measure whose moments through degree $2k-1$ exactly match those of $\mukfp$.
\end{restatable}

The majority of the remainder of this section is devoted to proving \cref{thm:backwards}.

As noted in \cref{rem:backwards_finite}, if $\mu = \muN$, $\mu_*$ is the VESD of a nearby problem $(\vec{A},\vec{b}_*)$, which is the \emph{same dimension} as the original problem $(\vec{A},\vec{b})$.
In this case $\momDel{k}(\muN,\muN;\muN) = 0$ and we have the following corollary:
\begin{corollary}
Under the assumptions of \cref{thm:backwards} (with $\mu=\mu_N)$, and assuming $\epslan$ is sufficiently small, there exists a nearby vector $\vec{b}^*$ satisfying
    \[
\|\vec{b} - \vec{b}_*\| \leq 2\sigma Dk^4 \opMax{k}(\mu_N;[\tta,\ttb])^2 \epslan,
    \]
    such that the Lanczos algorithm run on $(\vec{A},\vec{b}^*)$ for $k$ iterations in exact arithmetic produces $\mukfp$.
\end{corollary}
This implies backwards stability in the classical numerical linear algebra sense if $\opMax{k}(\muN;[\tta,\ttb])$ has polynomial growth in $k$.
If $\muN$ is sufficiently uniform relative to 
$k$ then the orthogonal polynomials through degree $k$ are easily shown to have polynomial growth using standard techniques; see for instance \cite[Lemma 6]{fischer_96}.
\begin{lemma}\label[lemma]{thm:op_bound}
Suppose, $\operatorname{supp}(\mu)\subseteq [\tta,\ttb]$ and, for some $K>0$ and $k\geq 1$, 
\begin{equation}
\label{eqn:regularity}
    \mu([x,y]) \geq K
    ,\qquad
    \forall x,y \in [\tta,\ttb]:
    |x-y| \geq (\ttb-\tta)/(16k^2).
\end{equation}
Then, 
\begin{equation*}
    \opMax{k}(\mu;[\tta,\ttb])
    \leq \frac{2}{\sqrt{K}}.
\end{equation*}
\end{lemma}

In some situations, the condition \cref{eqn:regularity} can be verified directly for $\mu = \muN$.
However, it will typically be easier to assume the Kolmogorov--Smirnov distance $ \KS(\muN,\mu_{\infty})$ between $\muN$ and some sufficiently regular measure $\mu_{\infty}$ is small.
\begin{assumption}[regularity of $\muN$]\label[assumption]{asm:nearby_regular}
Suppose $\mu_{\infty}$ is a measure with  support $[\tta,\ttb]$ such that, for some $L, \gamma > 0$,
\begin{equation}
    \label{eqn:asm_regular}
    \mu_{\infty}([x,y]) \geq L|x-y|^\gamma
    ,\qquad
    \forall x,y\in[\tta,\ttb],
\end{equation}
that for some $k\geq 1$,
\begin{equation*}
    \operatorname{supp}(\muN) \subseteq [\tta-(\ttb-\tta)/(32k^2),\ttb+(\ttb-\tta)/(32k^2)].
\end{equation*}
and that for some $\alpha>0$,
\begin{equation}
    \label{eqn:KS_assum}
    \KS(\muN,\mu_{\infty}) \leq N^{-\alpha}.
\end{equation}
\end{assumption}
In \cref{sec:random_matrix} we will discuss several common random matrix ensembles for which \cref{asm:nearby_regular} is satisfied in a probabilistic sense.

When \cref{asm:nearby_regular} is satisfied, the following result gives us a bound for $\opMax{k}(\muN;[\tta,\ttb])$.
\begin{corollary}\label[corollary]{thm:regularity_bounds}
Given \cref{asm:nearby_regular}, suppose
\begin{equation*}
    k \leq \sqrt{\frac{\ttb-\tta}{32}} \left( \frac{L N^{\alpha} }{3} \right)^{1/(2\gamma)}.
\end{equation*}
Then, for $[\tta',\ttb']:=[\tta-(\ttb-\tta)/(32k^2),\ttb+(\ttb-\tta)/(32k^2)]$,
\begin{equation*}    
\opMax{k}(\muN;[\tta',\ttb'])
    \leq  \frac{4}{\sqrt{L}} \left(\frac{32}{\ttb-\tta}\right)^{\gamma/2} k^\gamma.
\end{equation*}
\end{corollary}

\Cref{fig:GOE_kN_growth} shows the growth of the orthogonal polynomials $p_n(\cdot\,;\muN)$ with $n$ corresponding to the same random matrix model used in other figures.  
As expected, as $N$ increases, the degree $n$ for which the orthogonal polynomials of $\muN$ grow like those of $\mu_\infty$ increases.

\begin{figure}[hbt]
    \centering
    \includegraphics[width=\textwidth]{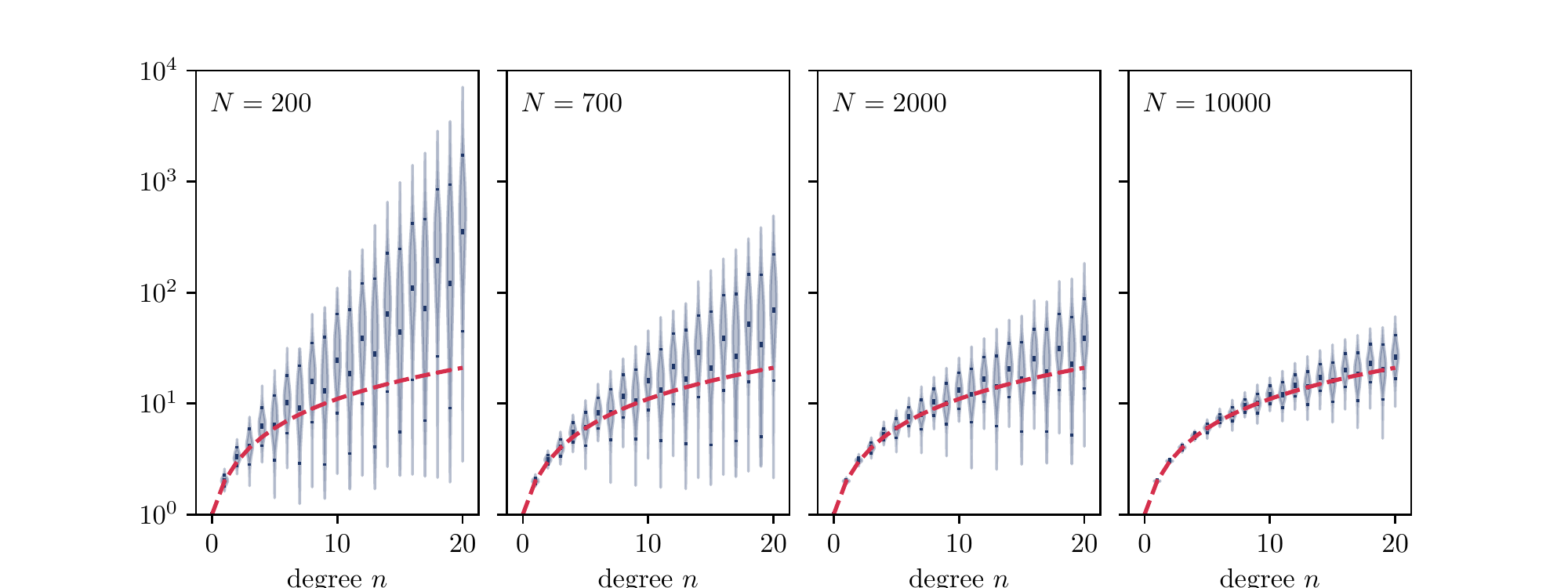}
    \caption{
    Maximum value of orthogonal polynomial $p_n(\cdot\,;\muN)$ over $[-1,1]$ and the maximum value in the $N \to \infty$ limit \textup{\textup{(\includegraphics[scale=.8]{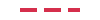})}}, where $\muN$ is drawn from the same random matrix model as in \cref{fig:motivating_experiment,fig:GOE_backwards}.
    For each $n,N$, the violin plot gives the distribution of $\|\op{n}(\cdot\,;\muN)\|_{[-1,1]}$, with the 5\%, 50\%, and 95\% quantiles marked explicitly.
    Note that for $k$ growing sufficiently slow with $N$, the maximum value of $p_n$ has polynomial growth for all $n\leq k$. 
    }
    \label{fig:GOE_kN_growth}
\end{figure}

Our forward stability analysis in \cref{sec:forwards} is based on a perturbation to a measure with a sufficiently nice density.
Assuming $\mu_\infty$ is sufficiently nice, we apply \cref{thm:backwards} with $\mu = \mu_\infty$.
This requires bounding $\opMax{k}(\mu_\infty;[\tta,\ttb])$ and $\momDel{k}(\mu_\infty,\muN;\muN)$.
\begin{corollary}\label[corollary]{thm:regularity_bounds_fwd}
Given \cref{asm:nearby_regular}, suppose that for some $c>0$
\begin{equation*}
    k \leq \bigg( \frac{\ttb-\tta}{16}\bigg)^{\gamma/ (4+2\gamma)} \bigg( \frac{c \sqrt{L} N^\alpha}{32} \bigg)^{1/(2+\gamma)}.
\end{equation*}
Then
\begin{equation*}
    \opMax{k}(\mu_\infty;[\tta,\ttb])
    \leq \frac{2}{\sqrt{L}} \left( \frac{16}{\ttb-\tta} \right)^{\gamma/2} k^{\gamma}
    ,\qquad
    \momDel{k}(\muN,\mu_\infty;\mu_\infty) \leq c.
\end{equation*}
\end{corollary}

The proofs of \cref{thm:op_bound,thm:regularity_bounds,thm:regularity_bounds_fwd} are given in \cref{sec:regular_proofs}.
Stronger bounds can be obtained in many situations. 
The stated bounds are simply meant to give a simple sufficient condition for the orthogonal polynomials to have polynomial growth with respect to $k$.

\subsection{Bounding the modified Chebyshev moments}
\label{sec:moments}

We will make frequent use of the well-known Chebyshev polynomials of the first and second kinds.
These families of polynomials are respectively defined by the recurrences
\begin{align*}
\pT_0(x) &= 1, & \pT_1(x) &= x, & \pT_{n}(x) &= 2x \pT_{n-1}(x) - \pT_{n-2}(x), & n&\geq 2, \\
\pU_0(x) &= 1, & \pU_1(x) &= 2x, & \pU_{n}(x) &= 2x \pU_{n-1}(x) - \pU_{n-2}(x), & n&\geq 2,
\end{align*}
and are respectively orthogonal with respect to the measures $\mu_{\pT}$ and $\mu_U$, each supported on $[-1,1]$, defined by
\begin{equation*}
\mu_{\pT}(\d{x}) :=  \frac{1}{\pi}\frac{1}{\sqrt{1-x^2}} \, \d{x}
,\qquad
\mu_{\pU}(\d{x}) :=  \frac{2}{\pi}\sqrt{1-x^2} \, \d{x}.
\end{equation*}
The Chebyshev polynomials of the first kind also satisfy the identities
\begin{equation}
\label{eqn:T_identities}
    \pT_{2n}(x) = 2 \pT_{n}(x)^2 - 1 
    ,\qquad
    \pT_{2n+1}(x) = 2 \pT_{n}(x) \pT_{n+1}(x) - x
    ,\qquad 
    n\geq0. 
\end{equation}

To prove \cref{thm:backwards}, it essentially suffices to show that the modified moments of the finite precision and exact arithmetic computations are near.
We begin by providing a slightly modified version of \cite{knizhnerman_96} for $\mu = \mu_{\pT}$ to allow for eigenvalues of $\vec{A}$ (and therefore $\Tkfp$) which may be just outside of $[-1,1]$.

It is well known that $\|\pT_{n}\|_{[-1,1]} \leq 1$ and $\|\pU_{n}\|_{[-1,1]}\leq n+1$.
Similar bounds hold on a slight extension of $[-1,1]$.
\begin{lemma}\label[lemma]{thm:poly_bd}
For any polynomial $p$ of degree $n$, with $\eta := 1/(2n^2)$,
    \begin{equation*}
        \|p\|_{[-1-\eta,1+\eta]} \leq 2 \| p\|_{[-1,1]}.
    \end{equation*}
\end{lemma}

This implies a bound for matrix Chebyshev polynomials of $\vec{A}$ and $\Tkfp$.
\begin{lemma}
\label[lemma]{thm:cheb_mat_bd}
Suppose that Lanczos is run on $(\vec{A},\vec{b})$ for $k\geq 1$ iterations with precision $\epslan<1/(5 k^{2}) $ and that 
$\|\vec{A}\| \leq 1+1/(4k^2)$.
Then, for all $n\leq k$, 
\begin{equation*}
    \| \pT_{n}(\vec{A}) \|, \| \pT_{n}(\Tkfp) \|
    \leq 2
    ,\qquad
    \| \pU_{n}(\vec{A}) \|, \| \pU_{n}(\Tkfp) \|
    \leq 2(k+1).
\end{equation*}
\end{lemma}

\begin{proof}
    For $k\geq1$, we have that $(1+1/(4k^2))(1 + \epslan) \leq 1 + 1/(2k^2)$.
    Thus, since $\|\vec{A}\| \leq 1+1/(4k^2)$, our assumption on $\epslan$ and \cref{def:lanczos_precision} imply $\Lambda(\Tkfp) \subseteq [-1-1/(2k^2),1+1/(2k^2)]$.
    The result follows by applying \cref{thm:poly_bd} and the fact $\|\pU_n\|_{[-1,1]} \leq n+1$, since the operator norm of a matrix function of a symmetric matrix is simply the maximum value of the function's absolute value evaluated at the eigenvalues of that matrix.
\end{proof}

We will also use the following fact about perturbed Chebyshev recurrences. 
This is a special case of a more general formula involving the associated polynomials of some family of orthogonal polynomials.
\begin{lemma}\label[lemma]{thm:assoc_p}
Suppose that
\begin{equation*}
\poly{d}_0 = 0,\qquad
\poly{d}_1 = f_0, \qquad
\poly{d}_n(x) = 2 x \poly{d}_{n-1}(x) - \poly{d}_{n-2}(x) + 2 f_{n-1}, \qquad n\geq 2.
\end{equation*}
Then, introducing the notation $U_{-1}(x) = 0$,
\begin{equation*}
    \poly{d}_n(x) = \pU_{n-1}(x) f_0 + 2 \sum_{i=2}^{n} \pU_{n-i}(x) f_{i-1}
    ,\qquad n\geq0.
\end{equation*}
\end{lemma}
\Cref{thm:poly_bd,thm:assoc_p} are proved in \cref{sec:thm:auxiliary_proofs}.

The next several results and the accompanying proofs follow \cite{knizhnerman_96} closely.
We include them so that our analysis is self-contained and in order to compute explicit constants.
In our proofs, for notational brevity, we define and use the vectors:
\begin{equation}
\label{eqn:short_vector_def}
    \vt{n} := \pT_{n}(\vec{A})\vec{b}
    ,\qquad \vtol{n} := \pT_{n}(\Tkfp) \vec{e}_0
    ,\qquad \vd{n} := \vt{n} - \Qkfp \vtol{n}
    ,\qquad
    \vr{n} := \Rk \vtol{n}.
\end{equation}

The first technical lemma we need is a bound on how well polynomials in $\vec{A}$ applied to $\vec{b}$ are approximated by the Lanczos quantities. 
To the best of our knowledge, a similar bound first appeared in \cite{druskin_knizhnerman_91} to analyze the behavior of the well-known Lanczos method for matrix function approximation; see also \cite{musco_musco_sidford_18}.

\begin{lemma}\label[lemma]{thm:TAe0}
Suppose that Lanczos is run on $(\vec{A},\vec{b})$ for $k>1$ iterations with precision $\epslan<1/(5k^2)$ and that
$\|\vec{A}\| \leq 1+1/(4k^2)$.
Then, for all $n\leq k-1$, 
\begin{equation*}
    \|\pT_{n}(\vec{A}) \vec{b} - \Qkfp \pT_{n}(\Tkfp) \vec{e}_0 \|
    \leq 9 k^{2} \epslan.
\end{equation*}
\end{lemma}

\begin{proof}
Since $k>1$, using the notation in \cref{eqn:short_vector_def} and recalling the perturbed recurrence \cref{eqn:lanczos_factorization_fp}, we have
\begin{equation*}
\vd{0} = \vec{b} - \Qkfp \vec{e}_0 = \vec{0}
,
\qquad
\vd{1} = \vec{A} \vec{b} - \Qkfp \Tkfp \vec{e}_0
= (\beta_k\vec{q}_{k-1} \vec{e}_{k-1}^\T + \Fk)\vec{e}_0
= \vec{F}_k \vtol{0}.
\end{equation*}
For \( n=2,\ldots, k-1 \), we can use the definitions of $\vt{n}$ and $\vtol{n}$, the definition of the Chebyshev polynomials, and the perturbed recurrence \cref{eqn:lanczos_factorization_fp} to write
\begin{align*}
    \vd{n} 
    &= (2\vec{A} \vt{n-1} - \vt{n-2}) - (2\Qkfp\Tkfp \vtol{n-1} - \Qkfp\vtol{n-2})
    \\&= 2(\vec{A} \vt{n-1} - (\vec{A} \Qkfp - \beta_k\vec{q}_{k-1} \vec{e}_{k-1}^\T - \Fk) \vtol{n-1}) - (\vt{n-2} - \Qkfp\vtol{n-2})
    \\&= 2(\vec{A} \vt{n-1} - (\vec{A} \Qkfp \vtol{n-1} - \beta_k\vec{q}_{k-1} \vec{e}_{k-1}^\T \vtol{n-1} - \Fk \vtol{n-1})) - \vd{n-2}.
\end{align*}
Note that $(\Tkfp)^i$ has half bandwidth $i$, so $(\Tkfp)^i$ is zero in the bottom left entry provided $i<k-1$.
Since $T_i$ is a degree $i$ polynomial, this implies that that $\vec{e}_{k-1}^\T\vtol{i} = \vec{e}_{k-1}^\T T_i(\Tkfp)\vec{e}_0 = 0$ for any $i<k-1$. 
Since $n<k$, applying this with $i=n-1$ we find
\begin{equation*}
    \vd{n} = 2\vec{A} \vd{n-1} - \vd{n-2} + 2 \Fk \vtol{n-1}.
\end{equation*}
\Cref{thm:assoc_p} with $x\to \vec{A}$, $d_n(x) \to \vd{n}$, and $f_n \to \vec{F}_k \vtol{n}$ allows us to obtain an explicit expression
\begin{equation}
\label{eqn:dn_rec}
    \vd{n} = \pU_{n-1}(\vec{A}) \Fk \vtol{0} + 2 \sum_{i=2}^{n} \pU_{n-i}(\vec{A}) \Fk \vtol{i-1}.
\end{equation}
Since $\epslan<1/(5k^2)$ and $\|\vec{A}\|\leq 1+1/(4k^2)$, if $\ell \leq k-1$, then \cref{thm:cheb_mat_bd} gives the bounds
\begin{equation*}
    \| \pU_{\ell}(\vec{A}) \| \leq 2k
    ,\qquad 
    \| \vtol{\ell} \|
    = \| \pT_{\ell}(\Tkfp) \|
    \leq 2.
\end{equation*}
Using \cref{def:lanczos_precision} and the assumption $k>1$, 
\[
\| \Fk \| \leq \| \vec{A} \| \epslan
\leq (1 + 1/(4k^2)) \epslan \leq (17/16) \epslan
< (9/8) \epslan.
\]
Finally, we apply the triangle inequality to \cref{eqn:dn_rec}, double the first term for convenience, apply the above bounds, and use the fact $n \leq k$ to obtain the bound
\begin{equation*}
    \| \vd{n} \| \leq 2 \sum_{i=1}^{n} \| \pU_{n-i}(\vec{A}) \| \|\Fk\| \| \vtol{i-1} \|
    \leq 2 n (2 k) ((9/8) \epslan) (2)
    < 9 k^2 \epslan.
    \qedhere
\end{equation*}
\end{proof}

If $\Qkfp$ had nearly orthonormal columns, we could use \cref{eqn:T_identities} to upgrade \cref{thm:TAe0} to a bound on the modified moments produced by the Lanczos algorithm. 
However, since we do not have such a guarantee, we require a bit more work.
We begin with a lemma akin to \cite[Lemma 1]{knizhnerman_96}.

\begin{lemma}\label[lemma]{thm:RTe0}
Suppose that Lanczos is run on $(\vec{A},\vec{b})$ for $k>1$ iterations with precision $\epslan<1/(5 k^{2}) $ and that
$\|\vec{A}\| \leq 1+1/(4k^2)$.
Then, for all $n\leq k-1$, 
\begin{equation*}
        \| \Rk \pT_{n}(\Tkfp) \vec{e}_0 \| 
        \leq 9 \| \vec{A} \| k^{2} \epslan.
    \end{equation*}
\end{lemma}

\begin{proof}
Since \( \Rk \) is strictly upper triangular, again using the notation in \cref{eqn:short_vector_def} and recalling the perturbed recurrence \cref{eqn:lanczos_R_factorization_fp}, we have 
\begin{equation*}
    \vr{0} = \vec{0}
    ,\qquad
    \vr{1} = \vec{R}_k \Tkfp \vec{e}_0 
    =  (\Tkfp \vec{R}_k - \overline{\beta}_{k-1} \Qkfp^\T \overline{\vec{q}}_k \vec{e}_{k-1}^\T  - \vec{H}_k)\vec{e}_0
    =  - \Hk \vtol{0}.
\end{equation*}
Analogous to the recurrence used in the previous proof, but now using the perturbed recurrence \cref{eqn:lanczos_R_factorization_fp}, for $n=2,\ldots, k-1$ the \( \vr{n} \) satisfy the perturbed three-term recurrence 
\begin{align*}
    \vr{n}
    &= 2 \Rk\Tkfp\vtol{n-1} - \Rk \vtol{n-2}
    \\&= 2(\Tkfp\Rk - \beta_{k-1} \Qkfp^\T \vec{q}_k \vec{e}_{k-1}^\T  - \Hk  )\vtol{n-1}  - \vr{n-2}
    \\&= 2(\Tkfp\Rk \vtol{n-1} - \beta_{k-1} \Qkfp^\T \vec{q}_k \vec{e}_{k-1}^\T \vtol{n-1} - \Hk  \vtol{n-1  })  - \vr{n-2}
    \\&= 2 \Tkfp \vr{n-1} - \vr{n-2}  - 2\Hk \vtol{n-1}.
\end{align*}
As above, using \cref{thm:assoc_p} with $x\to \Tkfp$, $d_n\to\vec{r}_n$, and $f_n \to -\vec{H}_k \vtol{n}$,
\begin{equation*}
    \vr{n} 
    = -\pU_{n-1}(\Tkfp)  \Hk \vtol{0} - 2 \sum_{i=2}^{n} \pU_{n-i}(\Tkfp) \Hk \vtol{i-1}.
\end{equation*}
As before, for $\ell \leq k-1$, \cref{thm:cheb_mat_bd} gives bounds
\begin{equation*}
    \| \pU_{\ell}(\Tkfp) \| \leq 2k
    ,\qquad 
    \| \vtol{\ell} \|
    = \| \pT_{\ell}(\Tkfp) \|
    \leq 2,
\end{equation*}
and \cref{def:lanczos_precision} and the assumption $k>1$ give the bound
\[
\| \Hk \| \leq \| \vec{A} \| \epslan
\leq (1 + 1/(4k^2)) \epslan \leq (17/16) \epslan
< (9/8) \epslan.
\]
We therefore obtain
\begin{equation*}
    \| \vr{n} \|
    \leq 2\sum_{i=1}^{n} \| \pU_{n-i}(\Tkfp) \| \| \Hk \| \| \vtol{i-1} \|
    \leq 2 n (2 k) ((9/8) \epslan) (2)
    < 9 k^2 \epslan.
    \qedhere
\end{equation*}
\end{proof}

We are now prepared to apply \cref{eqn:T_identities} to bound the modified Chebyshev moments.
For clarity, and following \cite[Lemmas 2 and 3]{knizhnerman_96}, we split this into a few steps.
Note that the maximal degree of the matrix-polynomials in the quadratic forms we analyze are $2k-2$. 
Owing to the fact that $\Tkfp$ is like a Jacobi matrix, one might expect the maximal degree should be $2k-1$, and indeed, in \cite{knizhnerman_96} a similar results for polynomials up to degree $2k-1$ is proved. 
This is not needed for our analysis.

\begin{lemma}\label[lemma]{thm:QTiTj}
Suppose that Lanczos is run on $(\vec{A},\vec{b})$ for $k>1$ iterations with precision $\epslan<1/(5 k^{2}) $ and that
$\|\vec{A}\| \leq 1+1/(4k^2)$.
Then, for all $m,n\leq k-1$, 
\begin{equation*}
    | \vec{b}^\T \pT_{m}(\Tkfp) \Qkfp^\T \Qkfp\pT_{n}(\Tkfp) \vec{b} - \vec{e}_0^\T \pT_{m}(\Tkfp) \pT_{n}(\Tkfp) \vec{e}_0 | \leq 37 k^{2} \epslan.
\end{equation*}
\end{lemma}

\begin{proof}
     Using the notation in \cref{eqn:short_vector_def} and the definitions of \( \Rk \) and $\Dk$,
    \begin{equation*}
        \vtol[\T]{m} \Qkfp^\T \Qkfp \vtol{n}
        =  \vtol[\T]{m} (\Rk+\Rk^\T + \vec{I} + (\Dk - \vec{I}))  \vtol{n}.
    \end{equation*}
    From \cref{def:lanczos_precision} we have $\| \Dk - \vec{I} \| \leq \epslan$.
    By assumption, $\|\vec{A}\|\leq 1+1/(4k^2)$ and $\epslan<1/(5k^2)$, so for all $\ell \leq k-1$, \cref{thm:cheb_mat_bd,thm:RTe0} respectively give bounds
    \[
    \|\vtol{\ell}\| = \| \pT_{\ell}(\Tkfp) \| \leq 2 
    ,\qquad 
    \|\vec{R}_k \vtol{\ell} \|
    = \| \vec{R}_k \pT_\ell(\Tkfp) \vec{e}_0\| \leq 9k^2\epslan.
    \]
    Combining these, we find, 
    \begin{align*}
        | \vtol{m} \Qkfp^\T \Qkfp \vtol{n} 
        - \vtol[\T]{m} \vtol{n} | \nonumber
        &\leq  \| \vtol{m} \| \| \Rk \vtol{n} \| + \| \vtol{n} \| \| \Rk \vtol{m} \| +  \| \Dk - \vec{I} \|  \| \vtol{m} \|  \| \vtol{n} \|
        \\&\leq 2(9k^{2}\epslan) + 2(9k^{2}\epslan) + \epslan (2)(2)
        \leq (36 k^{2} + 4)\epslan.
    \end{align*}
    Since $k>1$, $36k^2+4\leq 37k^2$.
\end{proof}

\begin{lemma}\label[lemma]{thm:TiTj}
Suppose that Lanczos is run on $(\vec{A},\vec{b})$ for $k>1$ iterations with precision $\epslan<1/(5 k^{2}) $ and that
$\|\vec{A}\| \leq 1+1/(4k^2)$.
Then, for all $m,n\leq k-1$, 
\begin{equation*}
    | \vec{b}^\T \pT_{m}(\vec{A}) \pT_{n}(\vec{A}) \vec{b} - \vec{e}_0^\T \pT_{m}(\Tkfp) \pT_{n}(\Tkfp) \vec{e}_0 | \leq 127 k^{2} \epslan.
\end{equation*}
\end{lemma}

\begin{proof}
    Using the notation in \cref{eqn:short_vector_def},
    \begin{equation*}
        \vt[\T]{m} \vt{n}
        = (\vd{m} + \Qkfp \vtol{m} )
^\T(\vd{n} + \Qkfp \vtol{n} )
        = \vd[\T]{m}\vd{n} + \vd[\T]{m} \Qkfp \vtol{n} 
        + \vtol[\T]{m} \Qkfp^\T \vd{n}
        + \vtol[\T]{m} \Qkfp^\T \Qkfp \vtol{n}.
    \end{equation*}
    Thus, applying the triangle inequality and submultiplicativity of the operator norm,
    \begin{align}
    \label{eqn:titj_titj_fp}
        | \vt[\T]{m} \vt{n} - \vtol[\T]{m} \vtol{n} |
        &\leq 
        | \vtol{m} \Qkfp^\T \Qkfp \vtol{n} 
        - \vtol[\T]{m} \vtol{n} |
        + \|\vd{m}\| \|\Qkfp\vtol{n}\| + \|\vd{n}\|\|\Qkfp \vtol{m}\| + \|\vd{m}\| \| \vd{n} \|.
    \end{align}
     By assumption, $\|\vec{A}\|\leq 1+1/(4k^2)$ and $\epslan<1/(5k^2)$, so for all $\ell \leq k-1$, \cref{thm:cheb_mat_bd,thm:TAe0} respectively give bounds
    \[
    \|\vtol{\ell}\| = \| \pT_{\ell}(\Tkfp) \| \leq 2 
    \qquad
    \| \vd{\ell} \| =
    \|\pT_{n}(\vec{A}) \vec{b} - \Qkfp \pT_{n}(\Tkfp) \vec{e}_0 \|
    \leq 9 k^{2} \epslan.
    \]
    This implies $9k^2 \epslan < 2$, so we find that
    \begin{equation*}
        \| \Qkfp \vtol{\ell} \|
        = \| \vd{\ell} + \vt{\ell} \|
        \leq \| \vd{\ell} \| + \| \vt{\ell} \|
        < 2 + 2
        = 4.
    \end{equation*}
    Under these same assumptions, \cref{thm:QTiTj} gives a bound 
    \[
    | \vtol{m} \Qkfp^\T \Qkfp \vtol{n} 
        - \vtol[\T]{m} \vtol{n} |
        \leq 37k^2 \epslan.
    \]
    Plugging the above bounds into \cref{eqn:titj_titj_fp} we find
    \begin{equation*}
        | \vt[\T]{m} \vt{n} - \vtol[\T]{m} \vtol{n} |
        \leq 37 k^{2} \epslan + 9 k^{2}\epslan (4) + 9 k^{2}\epslan(4) + (9 k^{2}\epslan)(2)
        \leq 127 k^2 \epslan. \qedhere
    \end{equation*}
\end{proof}

A bound for the modified moments with respect to the Chebyshev polynomials, akin to \cite[Theorem 1]{knizhnerman_96}, is now immediate.

\begin{theorem}\label{thm:cheb_moments_bd}
Suppose that Lanczos is run on $(\vec{A},\vec{b})$ for $k>1$ iterations with precision $\epslan < 1/(5  k^{2})$ and that
$\|\vec{A}\| \leq 1+1/(4k^2)$.
Then, for all $n\leq 2k-2$, 
\begin{equation*}
    \left| \int \pT_{n}(x) \muN(\d{x}) - \int \pT_{n}(x) \mukfp(\d{x}) \right|
    \leq 381 k^{2} \epslan.
\end{equation*}
\end{theorem}

\begin{proof}
    By definition, 
    \begin{equation*}
        \left| \int \pT_{n}(x) \muN(\d{x}) - \int \pT_{n}(x) \mukfp(\d{x}) \right|
        = | \vec{b}^\T \pT_{n}(\vec{A}) \vec{b} - \vec{e}_0^\T \pT_{n}(\Tkfp)  \vec{e}_0 |.
    \end{equation*}

    First, suppose $n=2i$ for $i \leq k-1$.
    As noted in \cref{eqn:T_identities}, $T_n(x) = 2 T_{i}^2(x) - 1$.
    By assumption $\vec{b}^\T \vec{b} = \vec{e}_0^\T \vec{e}_0 = 1$. 
    Therefore, we can apply \cref{thm:TiTj} to bound
    \begin{equation*}
        | \vec{b}^\T \pT_{n}(\vec{A}) \vec{b} - \vec{e}_0^\T \pT_{n}(\Tkfp)  \vec{e}_0 | 
        = | 2\vt[\T]{i} \vt{i} - 2\vtol[\T]{i} \vtol{i} | 
        \leq 2 (127 k^{2}\epslan) = 254 k^{2}\epslan.
    \end{equation*}
    
    Now, assume $n=2i+1$, $i\leq k-2$.
    Then $T_n(x) = 2T_i(x) T_{i+1}(x)-x$.
    Then, since $\pT_1(x) = x$, again using \cref{thm:TiTj},
    \begin{align*}
        | \vec{b}^\T \pT_{n}(\vec{A}) \vec{b} - \vec{e}_0^\T \pT_{n}(\Tkfp)  \vec{e}_0 | 
        &= | (2\vt[\T]{i} \vt{i+1} - \vec{b}^\T \vec{A} \vec{b}) - (2\vtol[\T]{i} \vtol{i+1} - \vec{e}_0^\T \Tkfp \vec{e}_0) | 
        \\&\leq 2| \vt[\T]{i} \vt{i+1} - \vtol[\T]{i} \vtol{i+1} |  + |\vec{b}^\T \vec{A}\vec{b} - \vec{e}_0^\T \Tkfp\vec{e}_0 |
        \\&\leq 2 (127 k^{2}\epslan) + 127 k^{2}\epslan
        = 381 k^{2} \epslan.
    \end{align*}
    The result follows.
\end{proof}

\subsection{General modified moments via a change of basis}

When the orthogonal polynomials $\op{n}(\cdot\,;\mu)$ have reasonable Chebyshev series, then a statement similar to \cref{thm:cheb_moments_bd} holds for the moments with respect to $\mu$.

\begin{corollary}[Stability of moments wrt. $\mu$]\label[corollary]{thm:moments}
Let $\muN$ be the VESD  for $(\vec{A},\vec{b})$ and $\mu$ a unit-mass measure with support contained in $[-1,1]$. 

Then, there exist absolute constants $C, D$ such that, whenever 
Lanczos is run on $(\vec{A},\vec{b})$ for $k\geq 1$ iterations with precision $\epslan < 1/(C k^{2})$ to produce $\mukfp$ and $\operatorname{supp}(\muN) \subseteq [-1-\eta,1+\eta]$ where $ \eta < 1/(16 k^2)$,
then 
$\momDel{k}(\muN,\mukfp;\mu) \leq D \opMax{k}(\mu;[-1,1]) k^{3} \epslan$.
\end{corollary}

\begin{proof}
Note that $\int \pT_{n}(x)^2 \mu_{\pT}(\d{x}) = 1/2$ for $n\geq 1$.
We can therefore decompose $\op{n}(x;\mu)$, $n\geq 0$ into Chebyshev polynomials of the first kind by
\begin{equation*}
    \op{n}(x;\mu)
    = c_{n,0} \pT_{0}(x) + c_{n,1} \pT_1(x) + \cdots + c_{n,n} \pT_{n}(x)
    ,
\end{equation*}
where the coefficients are obtained by
\begin{equation*}
    c_{n,0} := 
    \int \op{n}(x;\mu) \mu_{\pT}(\d{x})
    ,\qquad
    c_{n,i} := 2\int \op{n}(x;\mu) \pT_i(x) \mu_{\pT}(\d{x})
    ,\qquad 1\leq i \leq n
    .
\end{equation*}
Note that for all $n \leq 2k-1$ and $i\leq n$, since $\|T_i\|_{[-1,1]}\leq 1$,
\begin{equation}
    \label{eqn:coeff_bd}
    |c_{n,i}| \leq 2 \int \left| \op{n}(x;\mu)\right| |\pT_i(x)| \mu_{\pT}(\d{x})
    \leq 2 \opMax{k}(\mu;[-1,1]).
\end{equation}

Assuming $\epslan < 1/(5  (k+1)^{2})$ and $\operatorname{supp}(\muN) \subseteq [-1-\eta,1+\eta]$ where $ \eta < 1/(16k^2) \leq 1/(4 (k+1)^2)$ we can apply \cref{thm:cheb_moments_bd,eqn:coeff_bd} to get the bound, for $n \leq 2k-1$,
\begin{align*}
     | \mom{n}(\muN;\mu) - \mom{n}(\mukfp;\mu)\big| 
     &= \left| \sum_{i=1}^{n} c_{n,i} \left( \int \pT_i(x) \muN(\d{x}) - \int \pT_i(x) \mukfp(\d{x})   \right)  \right|
     \\&\leq \sum_{i=1}^{n} | c_{n,i} | \left| \int \pT_i(x) \muN(\d{x}) - \int \pT_i(x) \mukfp(\d{x})  \right|
     \\&
     \leq 4k \opMax{k}(\mu;[-1,1]) 381 (k+1)^{2} \epslan.
\end{align*}
Finally, since $k\geq 1$, $1/(5(k+1)^2) \geq 1/(20k^2)$W and $4 (381)k(k+1)^{2} \leq 6096 k^{3}$.
Setting $C=20$ and $D=6096$ establishes the result.
\end{proof}
Clearly a better bound could be obtained directly from the coefficients\footnote{The coefficients $c_{n,i}$ define a so-called connection coefficient matrix \cite{webb_olver_21}. Since we are always interested in an expansion in the Chebyshev polynomials of the first kind, we do not introduce this concept in generality.} $c_{n,i}$ rather than in terms of the maximum value of the $p_n(\cdot\,;\mu)$.
However, we are more interested in the existence of bounds which deteriorate with polynomials of $k$ rather than the precise dependencies on $k$, and the present approach results in slightly simpler statements and proofs.

\subsection{Proof of backwards stability}

We are now prepared to prove \cref{thm:backwards}. 
The approach is straightforward: transform $[\tta,\ttb]$ to $[-1,1]$ and then apply \cref{thm:moments} to get a bound for the moments. 
This will give us a bound on the size of $h$ in \cref{eqn:new_measure}.

\begin{proof}[Proof of \cref{thm:backwards}]

If $\|h\|_{[\tta,\ttb]} < 1$, then as described in \cref{sec:setup}, $\mu_*$ is a well-defined positive measure whose moments agree $\mukfp$ through degree $2k-1$.

Define
\begin{equation*}
    \hat{\vec{A}} := \frac{2}{\ttb-\tta} \vec{A} - \frac{\ttb+\tta}{\ttb-\tta} \vec{I}
    ,\qquad
    \hat{\overline{\alpha}}_i := \frac{2}{\ttb-\tta} \overline{\alpha}_i - \frac{\ttb+\tta}{\ttb-\tta},
\end{equation*}
\begin{equation*}
    \hat{\overline{\beta}}_i  := \frac{2}{\ttb-\tta} \overline{\beta}_i
    ,\qquad
    \Fkp  := \frac{2}{\ttb-\tta} \vec{F}_k
    ,\qquad
    \Hkp  := \frac{2}{\ttb-\tta} \Hk
    ,\qquad
    \hat{\eta}_k := \frac{2}{\ttb-\tta}\hyperref[eqn:eta]{{\color{black}\eta_k}}.
\end{equation*}
Then, 
\begin{equation*}
    \label{eqn:lanczos_factorization_fp_prime}
    \hat{\vec{A}} \Qkfp = \Qkfp \Tkfpp + \hat{\overline{\beta}}_{k-1} \overline{\vec{q}}_{k} \vec{e}_{k-1}^\T + \Fkp,
\end{equation*}
\begin{equation*}
    \label{eqn:lanczos_R_factorization_fp_prime}
    \Tkfpp \Rk = \Rk\Tkfpp  + \hat{\overline{\beta}}_{k-1}\Qkfp^\T \overline{\vec{q}}_k\vec{e}_{k-1}^\T + \Hkp ,
\end{equation*}
\begin{equation*}
    \Lambda(\Tkfpp) \subseteq [ \lmin(\hat{\vec{A}}) - \hat{\eta}_k, \lmax(\hat{\vec{A}}) + \hat{\eta}_k]
\end{equation*}
where
\begin{equation*} 
    \| \Fkp \| \leq \frac{2 \| \vec{A} \|}{\ttb-\tta} \epslan
    ,\qquad
     \| \Dk - \vec{I} \| \leq \epslan
    ,\qquad
    \| \Hkp \| \leq \frac{2 \| \vec{A} \|}{\ttb-\tta} \epslan
    ,\qquad
    \hat{\eta}_k \leq \frac{2 \| \vec{A} \|}{\ttb-\tta} \epslan.
\end{equation*}
Thus, $(\Tkfpp,\Qkfp)$ can be viewed as the output of the Lanczos algorithm run on $(\hat{\vec{A}},\vec{b})$ with precision
\begin{equation*}
    \epslanp:= \max\left\{ \frac{2\| \vec{A} \|}{\ttb-\tta} ,1 \right\} \epslan
    = \sigma\epslan.
\end{equation*}

Define
\begin{align*}
    t(x) = \frac{2}{b -a} \left( x - \frac{b+a}{2} \right), \qquad t([a,b]) = [-1,1],
\end{align*}
and let $\muNp$ and $\hat{\mu}$ be the pushforward measures of $\mu_N$ and $\mu$, respectively, under $t$.
That is, for any measurable function $f$,
\begin{equation*}
    \int f(x) \d\muN(x)
    = \int f(t(x)) \d\muNp(x)
    ,\qquad
    \int f(x) \d\mu(x)
    = \int f(t(x)) \d\hat{\mu}(x).
\end{equation*}
This implies modified moments of $\muNp$ with respect to the orthogonal polynomials of $\hat{\mu}$ are the same as those of $\muN$ with respect to $\mu$.
Indeed,
\begin{equation*}
\op{n}(t(x);\hat{\mu})
= \op{n}(x;\mu)
\end{equation*}
so
\begin{equation*}
    \mom{n}(\muN;\mu) 
    = \int \op{n}(x;\mu) \d\muN
    = \int \op{n}(t(x);\hat{\mu}) \d\muN
    = \int \op{n}(x;\hat{\mu}) \d\muNp
    = \mom{n}(\muNp;\hat{\mu}).
\end{equation*}
Moreover,
\begin{equation*}
    \operatorname{supp}(\muNp) \subseteq [-1-1/(16k^2), 1+1/(16k^2)].
\end{equation*}

The assumption on $\epslan$ ensures $\epslanp < 1/(Ck^2)$, so \cref{thm:moments} gives a bound
\begin{equation*}
    \momDel{n}(\mukfp,\muN;\mu) 
    = \momDel{n}(\mukfpp,\muNp;\hat{\mu}) 
    \leq D \opMax{k}(\mu;[\tta,\ttb]) k^3 \epslanp.
\end{equation*}
Since the moments $\mom{n}(\mu_*;\mu) = \mom{n}(\mukfp;\mu)$ for $n \leq 2k-1$, we get \cref{thm:backwards}\ref{thm:backwards:nearby_input}.
Using the triangle inequality we also have
\begin{align*}
    \momDel{n}(\mukfp,\mu;\mu)
    \leq \momDel{n}(\mukfp,\muN;\mu) + \momDel{n}(\muN,\mu;\mu).
\end{align*}
From \cref{eqn:h_bd} we have $\|\poly{h}\|_{[\tta,\ttb]} \leq 2k \momDel{k}(\mukfp,\mu;\mu) \opMax{k}(\mu; [\tta,\ttb])$.
Thus,
\begin{equation*}    
    \|\poly{h}\|_{[\tta,\ttb]} 
    \leq  2k \opMax{k}(\mu; [\tta,\ttb]) \big( \momDel{n}(\mukfp,\muN;\mu) + \momDel{n}(\muN,\mu;\mu) \big).
\end{equation*}
Replacing $\epslanp$ with $\sigma \epslan$ gives \cref{thm:backwards}\ref{thm:backwards:h_bd}.
\end{proof}

\section{Perturbation theory for recurrence coefficients}
\label{sec:forwards}

We will now show that $\Tkfp$ is close to $\Tk$, at least when $\muN$ is near a sufficiently nice measure $\mu_\infty$.
Our forward stability results, which are stated at the end of the section, are derived by applying \cref{thm:backwards} with $\mu = \mu_\infty$ and then using a Riemann--Hilbert approach to analyze perturbations to the orthogonal polynomials of $\mu_\infty$.

The conditioning of the map from the modified moments of a measure to the corresponding recurrence coefficients was studied in \cite{fischer_96}; see also \cite[Section 2.1.6]{gautschi_04}.
However, it is not immediately clear how to apply the formulas and bounds to the situation in which $k$ grows as $\epslan$ tends to zero.

To study such a scaling, we use an alternate approach.
As a matter of technical convenience, we assume that $\mu$ has a density with square root behavior at the edges.
\begin{assumption}\label[assumption]{asm:mu_sqrt}
Suppose that $\operatorname{supp}(\mu) = [\tta,\ttb]$ and that $\mu$ has density
\begin{equation*}
    \rho(z) = g(z) \sqrt{(\ttb-z)(z-\tta)},
\end{equation*}
where $g$ is positive on $[\tta,\ttb]$ and analytic on an open set that contains $[\tta,\ttb]$.
\end{assumption}

Note that through the obvious affine transformation, we now assume, without loss of generality, that $\operatorname{supp}(\mu) = [-1,1]$.  We show that when $\mu$ is perturbed slightly, the recurrence coefficients for the perturbed measure are near those of $\mu$.
This culminates in \cref{thm:forwards}, a forward stability result for the Lanczos algorithm given at the end of the section.

\newcommand{\ttaa}{-1}
\newcommand{\ttbb}{1}
Suppose $\mu$ satisfies \cref{asm:mu_sqrt}, $[\tta,\ttb] = [\ttaa,\ttbb]$ and let $\ppi_n(z;\mu)$ denote the $n$-th monic orthogonal polynomial. 
Consider the matrix-valued function
\begin{equation*}
    \vec{Y}_n(z;\mu) = 
    \begin{bmatrix}
    \ppi_n(z;\mu)  & c_n(z;\mu) \\
    \gamma_{n-1}(\mu) \ppi_{n-1}(z;\mu) & \gamma_{n-1}(\mu) c_{n-1}(z;\mu)
    \end{bmatrix}
    ,\qquad n\geq 1,
\end{equation*}
where
\begin{equation*}
    c_n(z;\mu) := \frac{1}{2\pi \ii} \int \frac{\ppi_n(x;\mu)}{x-z} \mu(\d{x})
    ,\qquad
    \gamma_{n-1}(\mu) := -2\pi \ii \| \ppi_n(\cdot\,;\mu) \|_{L^2(\mu)}^{-2}.
\end{equation*}
For a function $y$, analytic in $\mathbb C \setminus [\ttaa,\ttbb]$, define the boundary values
\begin{equation*}
y^\pm(z) := \lim_{\epsilon \downarrow 0} y(z \pm \ii \epsilon), \quad z \in [\ttaa,\ttbb],
\end{equation*}
provided this limit exists.

With this notation in mind, $\vec{Y}_n(z;\mu)$ has the following properties:
\begin{itemize}
\item $\displaystyle \vec{Y}_n^+(z;\mu) = \vec{Y}_n^-(z;\mu) 
    \begin{bmatrix}
    1 & {\rho(z)} \\ 0 & 1 
    \end{bmatrix}
    ,\qquad 
    z\in [\ttaa,\ttbb],$
\item $\displaystyle
    \vec{Y}_n(z;\mu) z^{-n \bm{\sigma}_3}
    = \vec{I} + O(z^{-1})
    ,\quad
    z\to\infty, \quad \bm{\sigma}_3 := \mathrm{diag}(1,-1)$,
\item $\displaystyle\det\vec{Y}_n(z;\mu) = 1,$ and
\item $\vec{Y}_n(z;\mu)$ is analytic in $\mathbb{C}\setminus[\ttaa,\ttbb]$.
\end{itemize}
This is the so-called Fokas--Its--Kitaev Riemann--Hilbert problem \cite{fokas_its_kitaev_92}:  
The problem of finding the sectionally analytic function $\vec{Y}_n$ from the stated conditions.
References \cite{kuijlaars_03,deift_00} provide a comprehensive introduction.  It has been used in many contexts, both computationally and asymptotically, see \cite{olver_trogdon_13,townsend_trogdon_olver_14,deift_00}, for example.  But it can also be used for perturbation theory \cite{ding_trogdon_21}.  The Riemann--Hilbert representation play a role similar to the contour integral representation of classical orthogonal polynomials and allows one to estimate quantities related to the polynomials (e.g., recurrence coefficients) under consideration.  And even in the classical cases, it can prove to be a more powerful tool.  Below we will use Riemann--Hilbert theory to estimate $Y_n(z;\mu)$ which, via $Y_n^1(\mu)$ below, gives estimates on the recurrence coefficients for orthogonal polynomials, i.e., the output of the Lanczos iteration. 

Note that the orthonormal polynomials $\op{n}$ satisfy \begin{equation*}
    \op{n}(z;\mu) = \frac{\ppi_n(z;\mu)}{\|\ppi_n(\cdot\,;\mu)\|_{L^2(\mu)}}, \qquad \|\ppi_n(\cdot\,;\mu)\|_{L^2(\mu)}^2 = \int \ppi_n(x;\mu)^2 \mu(\d x).
\end{equation*}
The recurrence coefficients for the orthonormal polynomials are $(\alpha_j(\mu))_{j\geq 0}$, $(\beta_j(\mu))_{j \geq 0}$ such that, with $\beta_{-1}(\mu) := 0 =: p_{-1}(x;\mu)$,
\begin{equation*}
    x \op{n}(x;\mu) = \alpha_n(\mu) \op{n}(x;\mu) + \beta_n(\mu) p_{n+1}(x;\mu) + \beta_{n-1}(\mu) p_{n-1}(x;\mu).
\end{equation*}

These recurrence coefficients can be obtained directly from $\vec Y_n(z;\mu)$ via the formulae \cite{deift_00}
\begin{align*}    
    \alpha_n(\mu) &= [\vec{Y}_n^{(1)}(\mu)]_{1,1} - [\vec{Y}_{n+1}^{(1)}(\mu)]_{1,1}, \\
    \beta_n^2(\mu) &= \frac{[\vec{Y}_{n+1}^{(1)}(\mu)]_{1,2}}{[\vec{Y}_n^{(1)}(\mu)]_{1,2}} = {[\vec{Y}_n^{(1)}(\mu)]_{1,2}[\vec{Y}_n^{(1)}(\mu)]_{2,1}},
\end{align*}
where $\vec{Y}_n^{(1)}$ is the unique, $z$-independent matrix such that
\begin{equation*}
    \vec{Y}_n(z;\mu) 
    z^{-n \bm{\sigma}_3}
    = \vec{I} + \vec Y_n^{(1)}(\mu) z^{-1} + O(z^{-2}),\qquad
    z\to\infty.
\end{equation*}
Let\footnote{In the general case where $[\tta, \ttb] \neq [-1,1]$, one would set $\mathfrak c = \frac{b-a}{4}$.}
\begin{equation*}
    \mathfrak c = \frac{1}{2}.
\end{equation*}
One considers
\begin{equation*}
    \vec S_n(z;\mu) = {\mathfrak c}^{-n \bm{\sigma}_3} {\vec Y}_n(z;\mu) \phi(z)^{-n \bm{\sigma}_3},
\end{equation*}
where $\phi(z) = z + (z^2 -1)^{1/2}$ and the branch of the square root is chosen such that it is analytic in $\mathbb C \setminus [-1,1]$ and positive for $z > 1$.   Then it follows that \cite{kuijlaars_mclaughlin_vanassche_vanlessen_04}
\begin{equation*}
    \vec S_n(z;\mu) = \vec I + \vec S_n^{(1)}(\mu) z^{-1}  + O(z^{-2}), \qquad z \to \infty,
\end{equation*}
where $\vec S_n^{(1)}(\mu) = O(1)$ as $n \to \infty$.   Then, we see that $\phi(z) = 2z( 1 +  O(z^{-2}))$.
From this we obtain the expressions
\begin{align*}
    \vec S_n^{(1)}(\mu) = {\mathfrak c}^{-n \bm{\sigma}_3}\vec Y_n^{(1)}(\mu){\mathfrak c}^{n \bm{\sigma}_3}, \quad 
    \vec Y_n^{(1)}(\mu) = {\mathfrak c}^{n \bm{\sigma}_3}\vec S_n^{(1)}(\mu){\mathfrak c}^{-n \bm{\sigma}_3}.
\end{align*}

We are concerned with the case where $\mu_*$ is a relative perturbation of $\mu$.  Suppose
\begin{equation}\label{eq:heps}
    \mu_*(\d{x}) = (1+\poly{h}(x))\rho(x) \d{x}, \quad |\poly{h}(x)| < 1.
\end{equation}
This encompasses our definition of $\mu_*$ in \cref{thm:backwards}.  Now, consider
\begin{equation*}
    \check{\vec Y}_n(z;\mu)
    = \mathfrak c^{-n \bm{\sigma}_3}
    \vec{Y}_n(z;\mu)
    ,\qquad
    \vec{X}_n(z;\mu,\mu_*)
    = \check{\vec{Y}}_n(z,\mu_*) \check{\vec{Y}}_n(z;\mu)^{-1}.
\end{equation*}

Then for $z \in [\ttaa,\ttbb]$,
\begin{align*}
    \vec{X}_n^+(z;\mu,\mu_*)
    &= \check{\vec{Y}}_n^+(z,\mu_*) \check{\vec{Y}}_n^+(z;\mu)^{-1}
    \\&=
    \check{\vec{Y}}_n^-(z,\mu_*)
    \begin{bmatrix}
    1 & (1+\poly{h}(z))\rho(z) \\ 0 & 1 
    \end{bmatrix}
    \check{\vec{Y}}_n^+(z;\mu)^{-1}
    \\&={
    \check{\vec{Y}}_n^-(z,\mu_*)
    \begin{bmatrix}
    1 & (1+\poly{h}(z))\rho(z) \\ 0 & 1 
    \end{bmatrix}
    \begin{bmatrix}
    1 & -\rho(z) \\ 0 & 1 
    \end{bmatrix}
    \check{\vec{Y}}_n^-(z;\mu)^{-1}}
    \\&=
    \check{\vec{Y}}_n^-(z,\mu_*)
    \begin{bmatrix}
    1 & \poly{h}(z)\rho(z) \\ 0 & 1 
    \end{bmatrix}
    \check{\vec{Y}}_n^-(z;\mu)^{-1}
    \\&={\vec{X}_n^-(z;\mu,\mu_*) \check{\vec{Y}}_n^-(z;\mu)\begin{bmatrix}
    1 & \poly{h}(z)\rho(z) \\ 0 & 1 
    \end{bmatrix}
    \check{\vec{Y}}_n^-(z;\mu)^{-1}}
    \\&=\vec{X}_n^-(z;\mu,\mu_*)
    \left[ \vec{I} + h(z)\rho(z)
    \check{\vec{Y}}_n^-(z;\mu) \begin{bmatrix}
    0 & 1 \\ 0 & 0
    \end{bmatrix}
    \check{\vec{Y}}_n^-(z;\mu)^{-1}
    \right].
\end{align*}

Define
\begin{align}
    \label{eqn:Mk_def}
    \vec{M}_n(z;\mu)
    &:=\check{\vec{Y}}_n^-(z;\mu) \begin{bmatrix}
    0 & 1 \\ 0 & 0
    \end{bmatrix}
    \check{\vec{Y}}_n^-(z;\mu)^{-1}
    \\\nonumber&=
    \begin{bmatrix}
    0 & {\mathfrak c}^{-n} \ppi_n(z;\mu) \\ 
    0 & {\mathfrak c}^{n} \gamma_{n-1}{(\mu)} \ppi_{n-1}(z;\mu)
    \end{bmatrix}
    \begin{bmatrix}
    {\mathfrak c}^{n} \gamma_{n-1}{(\mu)} c_{n-1}^-(z;\mu) & -{\mathfrak c}^{-n} c_n^-(z;\mu) \\ 
    -{\mathfrak c}^{n} \gamma_{n-1}{(\mu)} \ppi_{n-1}(z;\mu) & {\mathfrak c}^{-n} \ppi_n(z;\mu) \\ 
    \end{bmatrix}
    \\&=
    \begin{bmatrix}
    -\gamma_n{(\mu)}\ppi_n(z;\mu) \ppi_{n-1}(z;\mu) & {\mathfrak c}^{-2n} \ppi_n(z;\mu)^2 \\
    -{\mathfrak c}^{2n} \gamma_{n-1}^2{(\mu)} \ppi_{n-1}(z;\mu)^2 & \gamma_{n-1}{(\mu)} \ppi_{n-1}(z;\mu) \ppi_n(z;\mu)
    \end{bmatrix}.\nonumber
\end{align}

For a piecewise-smooth contour $\Gamma \subset \mathbb C$, and matrix-valued functions $\vec X: \Gamma \to \mathbb C^{2\times 2}$ we use the norm
\begin{align*}
    \|\vec X\|_{L^2(\Gamma)} := \left( \int_{\Gamma} \|\vec X(x)\|_{\F}^2 \,|\d x| \right)^{1/2},
\end{align*}
where $\|\cdot\|_{\F}$ denotes the Frobenius (Hilbert--Schmidt) norm on $\mathbb C^{2 \times 2}$ and $L^2(\Gamma)$ is used to denote the space of measurable functions such that this norm is finite. 
Define the Cauchy operator
\begin{equation*}
\mathcal C^- u(z) = \lim_{\epsilon \downarrow 0} \mathcal C u( z - \ii \epsilon), \quad \mathcal C u(z) = \frac{1}{2 \pi \ii} \int_{-1}^1 \frac{u(z')}{z'-z} \,\d z'.
\end{equation*}
On $L^2([\ttaa,\ttbb])$, $\mathcal C^-$ is bounded with norm one \cite{Bottcher1997}.

\begin{lemma}\label[lemma]{l:rhp}
Suppose \cref{asm:mu_sqrt} holds and that $h$ is as in \eqref{eq:heps}. If
$$\Delta(n) = \Delta(n;\mu, \mu_*):=\sup_{z \in [\ttaa,\ttbb]} \|\poly{h}(z) \rho(z)\vec M_n(z;\mu)\|_{\F}  < 1$$
then there exists $\vec u_n \in L^2([\ttaa,\ttbb])$ such that
\begin{equation*}
\vec X_n(z;\mu,\mu_*) = \vec I + \mathcal C\vec u_n (z)= \vec I + O\left(\frac{\Delta(n)}{1 + |z|}\right),  \quad \|\vec u_n\|_{L^2([\ttaa,\ttbb])} \leq \frac{\Delta(n)\sqrt{2}}{1 - \Delta(n)},
\end{equation*}
uniformly with respect to $z$ on closed subsets of $\mathbb C \setminus [\ttaa,\ttbb]$.
\end{lemma}
\begin{proof}
Due to the analyticity of $\vec X_n$, it follows from the theory of Hardy spaces \cite{Duren} that there exists $\vec u_n \in L^2(\operatorname{supp}(\mu))$ such that $ \vec X_n(z;\mu,\mu_*) = \vec I + \mathcal C\vec u_n (z)$.  Then the Plemelj-Sokhotski lemma gives
that $\vec u_n(z) = \vec X_n^+(z;\mu,\mu_*) - \vec X_n^-(z;\mu,\mu_*)$.  This is a solution of the singular integral equation
\begin{equation*}
\vec u_n - (\mathcal C^- \vec u_n) \poly{h} \rho\vec M_n = \poly{h} \rho\vec M_n,
\end{equation*}
where the operator
\begin{align}\label{eq:sie}
\vec u \mapsto \vec u - (\mathcal C^- \vec u) \poly{h} \rho\vec M_n,
\end{align}
is near identity if $\Delta(n) < 1$.  Thus $\vec u_n$ is the unique solution of this integral equation and
\begin{equation*}
\|\vec u_n\|_{L^2([\ttaa,\ttbb])} \leq \frac{\Delta(n)\sqrt{2}}{1 - \Delta(n)}.
\end{equation*}    
The claim then follows from the expression
\begin{equation*}
\vec X_n(z;\mu,\mu_*) = \vec I + \mathcal C \vec u_n(z).
\qedhere
\end{equation*}
\end{proof}

This lemma is useful when \cref{thm:backwards} applies.  And, as we do in the following section, it can be used to compare both $\vec T_k$ and $\overline{\vec T}_k$ to a limiting three-term recurrence relation satisfied by limiting measure $\mu_\infty$.  But \cref{thm:backwards} with $\epslan = 0$  and $ \mu = \mu_\infty$ gives
\begin{align}\label{eq:hinf}
    \|h_\infty\|_{[\ttaa,\ttbb]} \leq 2D\sigma \opMax{k}(\mu_\infty;[\ttaa,\ttbb])k \momDel{k}(\muN,\mu_\infty;\mu_\infty),
\end{align}
where $h_\infty$ is such that $(1 + h_\infty) \mu_\infty$ has its first $2k-1$ moments match with $\mu_N$.  And this allows one to estimate the behavior of the orthogonal polynomials with repsect to $\mu_N$.  In the context of \cref{l:rhp}, recall that, for simplicity, that $[a,b] = [-1,1]$ and that $\Gamma$ is a contour that encircles $[-1,1]$ that is a distance at least $\nu$ from $[-1,1]$.  Then if $\vec u_n(s)_{ij}$ denotes the $(i,j)$ entry of $\vec u_n(s)$,
\begin{align*}
\|\vec u_n(\cdot)_{ij}\|_{L^2([-1,1])} \leq \sqrt{2} \frac{\Delta(n,\mu_\infty,\mu_N)}{1 - \Delta(n,\mu_\infty,\mu_N)} \Rightarrow \|\mathcal C \vec u_n(\cdot)_{ij}\|_{L^\infty(\Gamma)} \leq \sqrt{2} \frac{\Delta(n,\mu_\infty,\mu_N)}{1 - \Delta(n,\mu_\infty,\mu_N)} \nu^{-1} |\Gamma|^{1/2},
\end{align*}
and therefore
\begin{align*}
    \|\vec X_n(\cdot;\mu_\infty,\mu_N)\|_{L^\infty(\Gamma)} \leq \sqrt{2} + \| \mathcal C \vec u_n \|_{L^\infty(\Gamma)} \leq \sqrt{2} + \frac{\sqrt{8 |\Gamma|}}{\nu} \frac{\Delta(n,\mu_\infty,\mu_N)}{1 - \Delta(n,\mu_\infty,\mu_N)}.
\end{align*}
And since $\det \vec X_N = 1$, the same estimate holds for $\vec X_N^{-1}$.  Now, consider \cref{thm:backwards} with $\mu = \mu_N$ and suppose that $\|h\|_{\Lambda(A)} \leq 1$.  We wish to again use the Fokas--Its--Kitaev Riemann--Hilbert problem, but now it must be modified to handle discrete weights.  The reformulation we use can be found in \cite{ding_trogdon_21}, and it involves a jump condition on a contour that encircles the support (with counter-clockwise orientation) of the measures under consideration.  Instead of the jump matrix being of the form
\begin{align*}
\begin{bmatrix} 1 & \rho(z) \\ 0 & 1 \end{bmatrix} 
\end{align*}
for $\mu(\d x) = \rho(x) \d x$ it becomes, for $z \in \Gamma$,
\begin{align*}
\begin{bmatrix} 1 & - \frac{1}{2\pi \ii} \int_{-1}^1 \frac{\mu(\d x)}{x-z} \\ 0 & 1 \end{bmatrix}.
\end{align*}

Define
\begin{align*}
    r_N(z) = \frac{1}{2 \pi \ii} \sum_{n=1}^{N} h(\lambda_n) \frac{( \vec{b}^\T \vec{u}_n )^2}{\lambda_n - z},
\end{align*}
and we note that $\mathfrak c = 1/2$.  If $\mu_N$ is the VESD for $(\vec A, \vec b)$ and $\mu_*$ is the VESD for $(\vec A, \vec b_*)$ as in \cref{rem:backwards_finite} we find
\begin{align*}
    \vec X_n^+(z;\mu_N,\mu_*) = \vec X_n^-(z;\mu_N,\mu_*) \left[ \vec{I} + r_N(z)
    \vec M_n(z;\mu_N) \right], \quad z \in \Gamma,
\end{align*}
where the $^\pm$ superscripts denote the limit to $\Gamma$ from the interior/exterior of $\Gamma$.  Then, we can write
\begin{align*}
\check{\vec{Y}}_n^-(z;\mu_N) = \vec X_n(z;\mu_\infty,\mu_N) \check{\vec{Y}}_n(z;\mu_\infty).
\end{align*}
Therefore
\begin{align*}
    \vec X_n^+(z;\mu_N,\mu_*) = \vec X_n^-(z;\mu_N,\mu_*) \left[ \vec{I} + r_N(z)
    \underbrace{\vec X_n(z;\mu_\infty,\mu_N) \vec M_n(z;\mu_\infty)\vec X_n(z;\mu_\infty,\mu_N)^{-1}}_{\vec M_n(z;\mu_N)}
    \right].
\end{align*}

We then use \cref{thm:backwards} to estimate $r_N$ on $\Gamma$ by
\begin{align*}
    \|r_N\|_{L^\infty(\Gamma)} \leq \frac{1}{2 \pi \nu} \|h\|_{\Lambda(A)} \leq \frac{k\opMax{k}(\mu;[-1,1])}{ \pi \nu}  (\momDel{k}(\mu_*,\muN;\muN)+\momDel{k}(\muN,\muN;\muN)).
    \end{align*}
And here $h$ is such that $\mu_* = (1 + h) \mu_N$ has the same moments as $\overline{\mu}_k$, the finite precision Lanczos measure for $n = 1,2,\ldots,2k-1.$

\begin{assumption}\label[assumption]{asm:rhoMn_bd}
Suppose \cref{asm:mu_sqrt} holds, let $\hat \mu_\infty$ be the push-forward measure of $\mu_\infty$ under $x \mapsto \frac{2}{b-a} \left( x - \frac{b + a}{2} \right)$, and let $\hat{\rho}$ be the corresponding density.
Suppose further that
\begin{enumerate}[label=(\alph*)]
    \item there exist $E,\delta, K > 0$ such that for all $k>K$, 
    $\max_{n\leq k+1}
\sup_{z\in[\tta,\ttb]} \| \hat{\rho}(z) \vec{M}_n(z;\hat{\mu}_\infty) \|_{\F} \leq E k^\delta$, and
\item $\Gamma$ is chosen such that there exist $E',\delta', K' >0$ such that for all $k>K'$, $$\max_{n\leq k+1}
\sup_{z\in \Gamma} \| \vec{M}_n(z;\hat{\mu}_\infty) \|_{\F} \leq E' k^{\delta'}, \quad \nu \geq \frac{1}{k^2}.$$
\end{enumerate}
\end{assumption}

\begin{remark}
    In this assumption, we want to be able to consider $\nu = O(k^{-2})$ because that is the scale on which orthogonal polynomials grow polynomially, see \cref{thm:poly_bd}.
\end{remark}

\begin{lemma}\label[lemma]{l:forward_main}
Fix $\mu_\infty$ with support $[-1,1]$ and constants $L,\alpha,E,E',\delta,\delta'$.
Suppose \cref{asm:nearby_regular} (with parameters $L,\gamma=3/2,\alpha$) holds for $\mu_N$ and $\mu_\infty$ and that $\mu = \mu_\infty$ satisfies \cref{asm:rhoMn_bd} (with parameters $E,E',\delta,\delta'$). 
    Suppose further that
    \[
    k = o(N^{2\alpha/(8+\delta)})
    ,\qquad
    \epslan = o(k^{-(9 + \delta')})
    ,\qquad N\to\infty.
    \]
    Then, for $\mu_*$ as in \cref{eq:heps},
    \begin{align*}
        \vec X_n(z;\mu_N,\mu_*) = \vec I + O\left( \frac{\epslan k^{8+\delta'}}{1 + |z|}\right)
        ,\qquad N\to\infty,
    \end{align*}
    uniformly on sets bounded away from $\Gamma$.
\end{lemma}

\begin{proof}
We point out that \cref{asm:mu_sqrt} forces $\gamma \geq 3/2$ and $\gamma=3/2$ is always possible.  

Recall that $\tta = -1$ and $\ttb = 1$.  If we set $c = N^{-\beta}$ and assume
\begin{equation}
\label{eqn:Xn_kcond}
    k \leq \bigg( \frac{1}{8}\bigg)^{\gamma/ (4+2\gamma)} \bigg( \frac{\sqrt{L} N^{\alpha-\beta}}{32} \bigg)^{1/(2+\gamma)},
\end{equation}
then \cref{thm:regularity_bounds_fwd} gives the bounds
\begin{equation*}    
    \opMax{k}(\mu_\infty;[-1,1])
    \leq \frac{2}{\sqrt{L}}  8^{\gamma/2} k^{\gamma}
    ,\qquad
    \momDel{k}(\muN,\mu_\infty;\mu_\infty) \leq N^{-\beta}.
\end{equation*}
With $\epslan = 0$, $\mu = \mu_\infty$, \Cref{thm:backwards} then implies
\begin{equation*}
    \| h_\infty \|_{[-1,1]} \leq  2 k \opMax{k}(\mu_\infty;[-1,1]) \momDel{k}(\muN,\mu_\infty;\mu_\infty).
\end{equation*}
Here $\sigma \leq 2$ since $\| \vec{A}_N \| \leq 2$ by \cref{asm:nearby_regular}.
Using \cref{asm:rhoMn_bd} we find a bound
\begin{equation}\label{eqn:foward_Delta}
\Delta(k; {\mu}_\infty, {\mu}_N) 
= 
\sup_{z\in [-1,1]} |{h}_\infty(z)| \|  {\rho}(z)\vec M_k(z;{\mu}_\infty)\|_{\infty} 
= O( k^{1+\gamma+\delta} N^{-\beta} ).
\end{equation}

Assuming $\epslan < 1/(\sigma Ck^2)$, $\mu = \mu_N$, we revisit
\Cref{thm:backwards} which then implies, 
\begin{equation*}
    \| h \|_{[-1,1]} \leq 2D\sigma\opMax{k}(\mu_N;[-1,1])^2 k^{4} \epslan.
\end{equation*}
Again $\sigma \leq 2$.  
Provided that $\Delta(n;\mu_\infty,\mu_N) < 1$, we estimate
\begin{align}
\|r_N \vec X_n(\cdot;\mu_\infty,\mu_N) \vec M_n(\cdot;\mu_\infty) \vec X_n(\cdot;\mu_\infty,\mu_N)^{-1} \|_{{L^\infty(\Gamma)}} \leq & \nonumber \\
 \left( 1 + \frac{2 \sqrt{|\Gamma|}}{\nu} \frac{\Delta(n,\mu_\infty,\mu_N)}{1 - \Delta(n,\mu_\infty,\mu_n)} \right)^2 &\frac{\|h\|_{[-1,1]}}{\pi \nu} \|\vec M_n(\cdot;\mu_\infty)\|_{L^\infty(\Gamma)}.\label{eq:r_N_est}
\end{align}
Since the Riemann--Hilbert problem has a jump condition on the contour $\Gamma$, \eqref{eq:sie} is replaced with
\begin{align*}
    \vec u \mapsto \vec u - (\mathcal C^-_{\Gamma} \vec u) r_N \vec M_n(\cdot;\mu_N).
\end{align*}
And to obtain a near-identity operator, we require that \eqref{eq:r_N_est} is small.  Here $\mathcal C_{\Gamma}$ is the Cauchy operator on the contour $\Gamma$ and it, operating on $L^2(\Gamma)$, may have a norm larger than one \cite{Bottcher1997} but we may choose $\Gamma$ so that the norm is bounded by 2.
Since $\gamma=3/2$, condition \cref{eqn:Xn_kcond} is stronger than
\[
k \leq \frac{1}{4} \left( \frac{LN^\alpha}{3} \right)^{1/(2\gamma)},
\]
which, with \cref{thm:regularity_bounds}, gives the bound
\begin{align*}
 \opMax{k}(\muN;[-1,1]) 
    \leq  \frac{4}{\sqrt{L}} 16^{\gamma/2} k^\gamma.
\end{align*}
Suppose 
\begin{equation}\label{eqn:forward_k_final}
    k = o(N^{\beta/(3 + \gamma + \delta)}), \qquad 
    \epslan = o( k^{-(\delta' + 2 \gamma + 6)}).
\end{equation}
Then, 
\[
\Delta(n,\mu_\infty,\mu_N) = \sup_{z\in [-1,1]} |{h}(z)| \|  {\rho}(z)\vec M_k(z;\muN)\|_{\infty} 
= O(k^{1+\gamma+\delta} N^{-\beta})
= o(1).
\]
The assumption on $k$ in \cref{eqn:forward_k_final} implies $k=o(N^\beta/(1+\gamma+\delta))$ so that, using \cref{eqn:foward_Delta},
\begin{align*}
    \frac{1}{\nu} \frac{\Delta(n,\mu_\infty,\mu_N)}{1 - \Delta(n,\mu_\infty,\mu_N)} 
    =  o(k^{3+\gamma+\delta} N^{-\beta}) = o(1),
\end{align*}
and therefore
\begin{align*}
    \text{\cref{eq:r_N_est}} = O(\nu^{-1} \epslan k^{\delta' + 2 \gamma + 4}) = o( \epslan k^{\delta' + 2 \gamma + 6}).
\end{align*}
We can conclude that
\begin{align*}
    \vec u_n - (\mathcal C^-_{\Gamma} \vec u_n) r_N \vec M_n(\cdot;\mu_N) = r_N \vec M_n(\cdot;\mu_N)
\end{align*}
has a unique solution which satisfies $\|\vec u_n\|_{L^2(\Gamma)} = O( \epslan k^{\delta' + 2 \gamma + 6})$ and $\vec X_n = \vec I +\mathcal C_{\Gamma} \vec u_n$. 

Recalling that $\gamma = 3/2$, to balance our constraints on $k$, we set
\begin{equation*}
    \beta = \frac{\alpha}{7} \left( \frac{16 + 2\delta}{7(9 + 2\delta)} \right)^{-1} = \alpha  \frac{ 9 + 2\delta}{16 + 2\delta}
    \quad
    \Longrightarrow 
    \quad
    \frac{\alpha - \beta}{2 + \gamma} = \frac{2 (\alpha - \beta)}{7} = \frac{2 \beta}{9 + 2\delta} = \frac{\beta}{3 + \gamma + \delta}.
\end{equation*}
Thus, both conditions relating $k$ and $N$ are satisfied $k = o(N^{2\alpha/(16+2\delta)}) = o(N^{\alpha/(8+\delta)})$.
\end{proof}

\subsection{Recurrence coefficients}

The residue at infinity of $\vec{S}_n$ can be computed by first defining $\vec{X}_n^{(1)}$ uniquely by
\begin{align*}
\vec X_n(z;\mu,\mu_*) = \vec I + \vec X_n^{(1)}(\mu,\mu_*) z^{-1} + O(z^{-2}), \quad z \to \infty,
\end{align*}
and using
\begin{align*}
    {\mathfrak c}^{n\bm{\sigma}_3} \check{\vec{Y}}_n(z;\mu_*)  z^{-n \bm{\sigma}_3}
    &= {\mathfrak c}^{n\bm{\sigma}_3} \vec X_n(z;\mu,\mu_*)\check{\vec{Y}}_n(z;\mu)  z^{-n \bm{\sigma}_3}\\
    &= {\mathfrak c}^{n\bm{\sigma}_3} \vec X_n(z;\mu,\mu_*){\mathfrak c}^{-n\bm{\sigma}_3} {\mathfrak c}^{n\bm{\sigma}_3}\check{\vec{Y}}_n(z;\mu)  z^{-n \bm{\sigma}_3}\\
    & =  \bigg( \vec I + {\mathfrak c}^{n\bm{\sigma}_3}\vec X_n^{(1)}(\mu,\mu_*){\mathfrak c}^{-n\bm{\sigma}_3} z^{-1} + O(z^{-2}) \bigg) 
    \\&\hspace{4em}\bigg( \vec I + z^{-1}\left[ {\mathfrak c}^{n \bm{\sigma}_3}\vec S_n^{(1)}(\mu){\mathfrak c}^{-n \bm{\sigma}_3} - \frac{\ttb + \tta }{2}n \bm{\sigma}_3\right] + O(z^{-2}) \bigg),
    \end{align*}
    and therefore
    \begin{equation*}
    \vec{Y}^{(1)}_n(\mu_*)  = {\mathfrak c}^{n\bm{\sigma}_3}\left( \vec X_n^{(1)}(\mu,\mu_*) + \vec S_n^{(1)}(\mu) - \frac{\ttb + \tta }{2}n \bm{\sigma}_3 \right){\mathfrak c}^{-n\bm{\sigma}_3}.
\end{equation*}
Then, we may express the recurrence coefficients for $\mu_*$ in terms of those for $\mu$, and $\vec X_n^{(1)}$, via
\begin{align*}
    \alpha_n(\mu_*) &= 
    \alpha_n(\mu) + [\vec{X}_n^{(1)}(\mu,\mu_*)]_{1,1} - [\vec{X}_{n+1}^{(1)}(\mu,\mu_*)]_{1,1},
    \\
    \beta_n(\mu_*)^2 &= {[\vec{X}_n^{(1)}(\mu,\mu_*) + \vec S_n^{(1)}(\mu) ]_{1,2}[\vec{X}_n^{(1)}(\mu,\mu_*) + \vec S_n^{(1)}(\mu)]_{2,1}} \\
    &= \beta_n(\mu)^2 + [\vec{X}_n^{(1)}(\mu,\mu_*)]_{1,2} [\vec S_n^{(1)}(\mu) ]_{2,1} + [\vec{X}_n^{(1)}(\mu,\mu_*)]_{2,1} [\vec S_n^{(1)}(\mu) ]_{1,2} 
    \\&\hspace{11em}+ [\vec{X}_n^{(1)}(\mu,\mu_*)]_{1,2} [\vec X_n^{(1)}(\mu,\mu_*) ]_{2,1}.
\end{align*}

To obtain the optimal scaling of quantities with respect to $\mathrm{supp}(\mu) = [a,b]$, it is convenient to rescale first so that the support of $\mu$ is transformed to $[-1,1]$, and then undo the scaling after estimates are obtained. 
 This gives a perturbation result.
\begin{corollary}\label[corollary]{thm:Tk_perturbation}
    Given the assumptions of \cref{l:rhp}, let $\hat \mu$ and $\hat \mu_*$ be the push-forward measures of $\mu$ and $\mu_*$ under $x \mapsto \frac{2}{b-a} \left( x - \frac{b + a}{2} \right)$, respectively.  Suppose that $\Delta(n)=\Delta(n;\hat \mu, \hat \mu_*)$ satisfies $\Delta(n), \Delta(n+1) \leq 1/2$. Then
    \begin{align*}
        |\alpha_n(\mu_*) - \alpha_n(\mu)| &\leq \pi^{-1}(\Delta(n) + \Delta(n+1)) (b-a),\\
        |\beta_n(\mu_*)^2 - \beta_n(\mu)^2| &\leq 2\pi^{-1}\left[\max_n \|\vec S_n^{(1)}(\hat \mu)\| \right] \Delta(n) (b-a)^2 + \pi^{-2} \Delta(n)^2 (b-a)^2.
    \end{align*}

\end{corollary}

\begin{proof}
   \cref{l:rhp} applies. 
    We work with the expression
    \begin{align*}
        \vec X_n^{(1)}(\hat \mu,\hat \mu_*) = - \frac{1}{2 \pi \ii}\int_{-1}^1 \vec u_n(s) \,\d s.
    \end{align*}
    Then if $\vec u_n(s)_{ij}$ denotes the $(i,j)$ entry of $\vec u_n(s)$,
    \begin{align}
        \|\vec u_n(\cdot)_{ij}\|_{L^1([-1,1])} \leq \sqrt{2} \|\vec u_n(\cdot)_{ij}\|_{L^2([-1,1])},
    \end{align}
    and therefore
    \begin{align*}
        \|\vec X_n^{(1)}(\hat \mu, \hat \mu_*)\|_{\max} \leq \frac{1}{\pi} \frac{\Delta(n)}{1 - \Delta(n)}.
    \end{align*}
    The same estimates hold with $n$ replaced with $n +1$.  The claim follows by recalling that $\max_n \|\vec S_n^{(1)}(\hat \mu)\| < \infty$ and that
    \begin{align*}
        \alpha_n(\mu) &= \frac{b-a}{2} \alpha_n(\hat \mu) + \frac{b+a}{2}, & \alpha_n(\mu_*) &= \frac{b-a}{2} \alpha_n(\hat \mu_*) + \frac{b+a}{2},\\
        \beta_n(\mu) &= \frac{b-a}{2} \beta_n(\hat \mu), & \beta_n(\mu_*) &= \frac{b-a}{2} \beta_n(\hat \mu_*). \qedhere
    \end{align*}
\end{proof}

We now establish our forward stability result.
\begin{theorem}[Forward stability]\label{thm:forwards}   
    Suppose $\mathrm{supp}(\mu_\infty) = [a,b]$ and let
    \begin{align*}
    \Gamma = \{ z \in \mathbb C: 1 + k^{-1} = |z + \sqrt{z-1}\sqrt{z+1}| \}.
    \end{align*}
    Suppose the assumptions\footnote{Note that within the assumptions of \cref{l:forward_main} is Assumption~\ref{asm:rhoMn_bd} which involves $\Gamma$.} of \cref{l:forward_main} hold for $\Gamma$ and $\hat \mu_\infty, \hat \mu_N$, the push-forward measures of $\mu_\infty, \mu_N$, respectively, under $x \mapsto \frac{2}{b-a} \left( x - \frac{b + a}{2} \right)$. 
    Suppose further that
    \[
    k = o(N^{\alpha/(8+\delta)})
    ,\qquad
    \epslan = o(k^{-(9 + \delta')})
    ,\qquad N\to\infty.
    \]
    Then
    \begin{align*}
        \max_{n\leq k+1}|\alpha_n(\mu_N) - \alpha_n(\overline{\mu}_k)| &= O \left( (b-a)\epslan k^{\delta' + 9}\right), \quad\text{and}
        \\
        \max_{n\leq k+1}|\beta_n(\mu_N)^2 - \beta_n(\overline{\mu}_k)^2| &= O \left( (b-a)^2\epslan k^{\delta' + 9}\right), \quad N\to\infty.
    \end{align*}
    Therefore, if $((\vec A_N, \vec b_N))_{N \geq 1}$ is a sequence of problems such that $\hat \mu_N$ ($\mu_N = \mu_{\rm{VESD}}(\vec A_N, \vec b_N)$) satisfies \cref{asm:nearby_regular},
    \begin{align*}
        \|\vec T_k - \overline{\vec T}_k\|_{\max} = O((b-a)\epslan k^{9 + \delta'}).
    \end{align*}
\end{theorem}

\begin{proof}
    Note that
    \begin{align*}
        \alpha_n(\muN) &= \frac{b-a}{2} \alpha_n(\muNp) + \frac{b+a}{2}, & \alpha_n(\mukfp) &= \frac{b-a}{2} \alpha_n(\mukfpp) + \frac{b+a}{2},\\
        \beta_n(\muN) &= \frac{b-a}{2} \beta_n(\muNp), & \beta_n(\mukfp) &= \frac{b-a}{2} \beta_n(\mukfpp).
    \end{align*}
    The only portion of this result that does not follow immediately from \cref{l:forward_main} is the last claim because one needs to show
    \begin{align*}
        |\beta_n(\muN) - \beta_n(\mukfp)| =  (b-a)\frac{|\beta_n(\muNp)^2 - \beta_n(\mukfpp)^2|}{|\beta_n(\muNp) + \beta_n(\mukfpp)|} =   O((b-a)|\beta_n(\muNp)^2 - \beta_n(\mukfpp)^2|). 
    \end{align*}
    Let $\hat{\mu}_\infty$ be the pushforward measure of $\mu_\infty$ under the same mapping as above.
    \Cref{thm:Tk_perturbation} gives a lower bound on $|\beta_n(\muNp) + \beta_n(\mukfpp)|$ using the fact that $ \beta_n(\muNp) \to \beta_n(\hat{\mu}_\infty)$ where the positivity of $\beta_n(\hat{\mu}_\infty)$ is crucial.  The restrictions on $k$ for this fact are milder than the assumptions of the theorem.  And then $\vec S_n^{(1)}(\muNp)$ is estimated using the relation
   \begin{align*}
       \vec S_n^{(1)}(\muNp) = \vec S_n^{(1)}(\hat{\mu}_\infty) + \vec X_n^{(1)}(\mu_\infty,\muNp),
   \end{align*}
   and using \cref{l:rhp}.
\end{proof}

\begin{remark}
    Our approach yields a forwards stability result for the Lanczos algorithm in the case that $\muN$ is near to a measure $\mu$ satisfying \cref{asm:mu_sqrt}.
    This is the case in many situations, for instance, when $\muN$ is the VESD associated to many large random matrices and $\mu$ is the limiting measure.
    In particular, in the context of \cref{fig:motivating_experiment} our approach explains the observation that $\overline{\alpha}_i$ and $\overline{\beta}_i$ are near $0$ and $1/2$ respectively. 
\end{remark}

\section{Random matrices}
\label{sec:random_matrix}

A natural setting in which \cref{asm:nearby_regular} holds is when $\mu_N = \muVESD(\cdot\,;\vec A, \vec b)$ is the VESD of a random matrix and (random or deterministic) vector and $\mu_\infty$ is the limiting spectral distribution.  In what follows $\mathbb E$ and $\mathbb P$ will denote the expectation and the probability of an event with respect to a probability distribution that will be clear from context.

We first define the notion of a \emph{local law}.  This holds for a wide class of random matrices \cite{knowles_yin_17,erdos_yau_17} and we suppose the limiting measure $\mu_\infty$ has bounded support.  We discuss two classical examples of such random matrices below.

Define two $N$-dependent regions in the complex plane by
\begin{align*}
    \mathcal D(N,\tau) &= \{ z \in \mathbb C : N^{-1 + \tau} \leq \imag z \leq \tau^{-1},~~ \mathrm{dist}(\real z, \operatorname{supp}\mu_\infty) < \tau^{-1} \},\\
    \mathcal D_o(N,\tau) &= \{ z \in \mathbb C : 0 < \imag z \leq \tau^{-1},~~ N^{-2/3 +  \tau} \leq \mathrm{dist}(\real z, \operatorname{supp}\mu_\infty) < \tau^{-1}  \}.
\end{align*}
The first region is useful for estimating quantities near the interior of the support of $\mu_\infty$ and the latter is used to estimate quantities near the edges of the support of $\mu_\infty$.

\begin{definition}
Suppose $\vec A_N$ is an $N \times N$ random matrix and $\vec b_N$ is an $N$-dimensional vector.  The sequence $\left((\vec A_N, \vec b_N)\right)_{N \geq 1}$, $\|\vec b_N\| = 1$ is said to satisfy a local law with limit $\mu_\infty$ if (1) for every fixed $\tau > 0$, $\epsilon > 0$ and $D > 0$ there exists $C$ such that
\begin{align}\label{eq:locallaw}
\begin{split}
 &\sup_{z \in \mathcal D(N,\tau)} \mathbb P \left(\left| \vec b_N^\T (\vec A_N - z \vec I)^{-1} \vec b_N - \mathcal S(z;\mu_\infty)\right| \geq  N^{\epsilon} \left( \sqrt{ \frac{\imag \mathcal S(z;\mu_\infty)}{N \imag z}} + \frac{1}{N \imag z} \right) \right)\leq C N^{-D},\\
 &\sup_{z \in \mathcal D_o(N,\tau)} \mathbb P \left(\left| \vec b_N^\T (\vec A_N - z \vec I)^{-1} \vec b_N - \mathcal S(z;\mu_\infty)\right| \geq  N^{\epsilon} \sqrt{ \frac{\imag \mathcal S(z;\mu_\infty)}{N \imag z}} \right)\leq C N^{-D}
 \end{split}
\end{align}
where
\begin{equation*}
    \mathcal S(z;\mu_\infty) = \int \frac{\mu_\infty(\d x)}{x - z},
\end{equation*}
and (2) there exists $L > 0$ such that
\begin{equation*}
\mathbb P( \|\vec A_N\| > L ) \leq C N^{-D}. 
\end{equation*}
\end{definition}

We pause to note that, importantly,
\begin{equation*}
    \vec b_N^\T (\vec A_N - z \vec I)^{-1} \vec b_N = \mathcal S(z; \muVESD(\vec A_N, \vec b_N)).
\end{equation*}
For a function $F$ of bounded variation we also use the notation
\begin{equation*}
    \mathcal S(z;F) := \int \frac{\d F(x)}{x - z}, \qquad \imag z >0,
\end{equation*}
to denote the Riemann--Stieltjes integral.

To turn a local law into an estimate on the KS distance we use the following.
\begin{corollary}[\cite{bai_silverstein_98},~Corollary B.15]\label[corollary]{c:bai}
    Let $F$ be a distribution function and let $G$ be a function of bounded variation satisfying $\int |F(x) - G(x)| \d x < \infty$.  Assume that, for some constants $A > B > 0$,
    \begin{equation*}
        \int_{-B}^B \d F(x) = 1, \qquad \int_{-\infty}^{-B} |\d G(x)| =  0 = \int_{B}^{\infty} |\d G(x)|.
    \end{equation*}
    Then
    \begin{align*}
        d_{\rm{KS}}(F,G) &:= \sup_{x} |F(x) - G(x)| 
        \leq \frac{1}{\pi (1 - \kappa) (2 \gamma -1)} \left[ \int_{-A}^A |\mathcal S(z;\d F) - \mathcal S(z;\d G)| \d u \right. 
        \\ &\hspace{18em}+ \left. \frac{1}{v} \sup_x \int_{|y| \leq 2 v a} |G(x + y) - G(x)| \,\d y \right],
    \end{align*}
    where $\kappa$ satisfies
    \begin{equation*}
        \kappa = \frac{4B}{\pi (A-B)(2 \gamma -1)} < 1,
    \end{equation*}
    $z = u + \ii v$, and $\gamma$ and $a$ are related via
    \begin{equation*}
        \gamma = \frac{1}{\pi} \int_{-a}^a \frac{\d u}{1 + u^2} > \frac 1 2.
    \end{equation*}
\end{corollary}

Without loss of generality, to estimate a KS distance, we can suppose that $\mathrm{supp}(\mu_\infty) \subseteq [-b,b]$.   
For $0 < \tau < 1$, set $A = b + \tau^{-1}, B = b + \tau$.  
Then
\begin{equation*}
    \kappa = \frac{ 4 (b + \tau)}{\pi (\tau^{-1}-\tau) (2 \gamma -1)} \overset{\tau \to 0}{\longrightarrow} 0.
\end{equation*}
And we choose $\tau$ sufficiently small so that this quantity $\kappa$ is less than 1. We assume $\tau$, and hence $A$, are chosen in this way for the forthcoming results.

\begin{lemma}\label[lemma]{l:bai-est}
    Suppose $\mathrm{supp}(\mu_\infty) \subseteq [-b,b]$, $b > 0$, where $\mu_\infty$ has a H\"older continuous, bounded density on $\mathbb R$.  Suppose the sequence $((\vec A_N, \vec b_N))_{N \geq 1}$ satisfies a local law with limit $\mu_\infty$.  Then for every $\epsilon > 0$, $D > 0$, $M > 1$, $z = u + \ii v \in \mathcal \mathcal D(N,\tau)$,
    \begin{align*}
      \mathbb P \left( \int_{-A}^A | \mathcal S(z; \muN) - \mathcal S(z;\mu_\infty)| \d u \geq N^{\epsilon} \sqrt{ \frac{1}{N \imag z}} + \frac{8 A^2}{(\imag z)^2 M} \right) \leq C M N^{-D}.
    \end{align*}
    where $\muN(\d x) = \muVESD(\d x; \vec A_N, \vec b_N)$.  Furthermore, if $G$ is the cumulative distribution function for $\mu_\infty$ then
    \begin{align*}
        \frac{1}{v} \sup_x \int_{|y| \leq 2 v a} |G(x + y) - G(x)| \,\d y \leq \sup_x |G'(x)| 2 v a^2.
    \end{align*}
\end{lemma}
\begin{proof}
    The last statement follows immediately after using the mean-value theorem.  To establish the first claim, we must discretize the integral that is involved and show that we need only use polynomially many discretization points to approximate it to any desired accuracy.  To do this, we show that the Lipschitz constant of the integrand depends on $\imag z$ in a sufficiently benign way.  For $\imag z > 0$,
    \begin{equation*}
        \left|\frac{\d}{\d z} \mathcal S(\mu_\infty;z)\right| = \int \frac{\mu_\infty(\d x)}{|x - z|^2} \leq \frac{1}{(\imag z)^2}. 
    \end{equation*}
    From this, it follows for $f(z) : =  \mathcal S(z; \muVESD(\vec A_N, \vec b_N)) - \mathcal S(z;\mu_\infty)$ we have
    \begin{equation*}
    \big| |f(u + \ii v)| - |f(u' + \ii v)| \big| \leq  |f(u + \ii v) - f(u' + \ii v)|  \leq \frac{2}{v^2} |u - u'|, \quad v > 0.
    \end{equation*}
    This implies that if one discretizes the integral
    \begin{equation*}
    \int_{-A}^A |f(u + \ii v)| \, \d u,
    \end{equation*}
    using $M+1$ equally-spaced points $-A = x_0, x_1, \ldots, x_M  = A$, then
    \begin{equation*}
        \Bigg| \frac{2A}{M} \sum_{j=0}^{M-1} |f(x_j + \ii v)| - \int_{-A}^A |f(u + \ii v)| \,\d u \Bigg| \leq \frac{8 A^2}{v^2 M},
    \end{equation*}
    because
    \begin{equation*}
        \left| \int_{x_j}^{x_{j+1}} (|f(u + \ii v)| - |f(x_j + \ii v)|) \d u \right| \leq \frac{2 A}{M} \frac{2}{v^2} |x_j - x_{j+1}| = \frac{2}{v^2} \left( \frac{2 A}{M} \right)^2.
    \end{equation*}
    Now set $v = N^{-\alpha}$ for $0 \leq \alpha < 1/2$ and fix $\epsilon > 0$.  And let $E_{N,M}$ be the event on which
    \begin{align*}
    \left| \vec b_N^\T (\vec A_N - z_j \vec I)^{-1} \vec b_N - \mathcal S(z_j;\mu_\infty)\right| \geq  N^{\epsilon} \left( \sqrt{ \frac{\imag \mathcal S(z_j;\mu_\infty)}{N \imag z_j}} + \frac{1}{N \imag z_j} \right),\quad
    z_j = x_j + \ii N^{-\alpha}, 
    \end{align*}
    for some $j = 0,1,2,\ldots,M-1$.  Using the estimates for $\mathcal D(N,\tau)$, we have that $\mathbb P(E_{N,M}) \leq C M N^{-D}$ where $D$ is as large as we like.  The result then follows by simply using that $\imag \mathcal S(z;\mu_\infty)$ is bounded in the upper-half plane and then bounding
    \begin{align*}
       \left( \sqrt{ \frac{\imag \mathcal S(z_j;\mu_\infty)}{N \imag z_j}} + \frac{1}{N \imag z_j} \right) \leq K  \sqrt{ \frac{1}{N v}},
    \end{align*}
    for a constant $K$.     
\end{proof}

\begin{theorem}
Suppose the sequence $((\vec A_N, \vec b_N))_{N \geq 1}$ satisfies a local law with limit $\mu_\infty$ and that $\mu_\infty$ has a H\"older continuous, bounded density on $\mathbb R$.  Then  for any $D>0$, $\epsilon > 0$, there exists $C> 0$ such that
\begin{equation*}
    \mathbb P \left( d_{\rm{KS}}(\muVESD(\vec A_N, \vec b_N),\mu_\infty) \geq N^{\epsilon -1/3} \right)\leq C N^{-D},
\end{equation*}
for $N$ sufficiently large.
\end{theorem}
\begin{proof}
So, using both \cref{c:bai} and \cref{l:bai-est} if $\mu_\infty$ has a H\"older continuous, bounded density on $\mathbb R$ and if the sequence $((\vec A_N, \vec b_N))_{N \geq 1}$ satisfies a local law with limit $\mu_\infty$, we can conclude that
\begin{equation*}
    \mathbb P \left( d_{\rm{KS}}(\muVESD(\vec A_N, \vec b_N),\mu_\infty) \geq N^{\epsilon} \sqrt{ \frac{1}{N \eta}} + \frac{8 A^2}{\eta^2 M} + 2 \|G'\|_\infty a^2 \eta \right) \leq C M N^{-D},
\end{equation*}
where all the constants have the same meaning as in \cref{c:bai} and \cref{l:bai-est}. To optimize the error here, we see that one should take $\eta = N^{-1/3}$ and $M = N$, establishing the claim.
\end{proof}

\begin{remark}
    The portion of the local law that applies to $\mathcal D_o$ can be used to show that for any $\epsilon >0$ the support points of the VESD for $(\vec A_N, \vec b_N)$ must lie within a distance $N^{-2/3 +\epsilon}$  of $\mathrm{supp}\,\mu$ with overwhelming probability \cite{erdos_yau_17}.
\end{remark}

\subsection{Wigner matrices}
\label{sec:wigner}

Consider a random matrix 
\begin{equation}
    \vec A_N = \frac{1}{2 \sqrt{N}} (a_{ij})_{1 \leq i,j\leq N}
\end{equation}
where the real-valued random variables $a_{ij}$ are jointly independent for $i \leq j$ and satisfy
\begin{equation*}
    \mathbb E [a_{ij}]= 0, \qquad \mathbb E [a_{ij}^2]= \begin{cases} 1 & i \neq j \\ c & i = j \end{cases}, \qquad a_{ij} = a_{ji},
\end{equation*}
and
\begin{equation*}
    \mathbb E [|a_{ij}|^k] = C_k < \infty,
\end{equation*}
for all $k > 2$. Such a matrix is called a Wigner matrix and the distribution is referred to as a Wigner ensemble. In \cite{erdos_yau_17}, for example, it is shown that for any sequence of vectors $(\vec b_N)_{N \geq 1}$ that are independent\footnote{This reference actually show that the local law holds for any fixed deterministic sequence of vectors, but the estimates the authors give are uniform in the choice of vectors, and the result can be extended to hold for random vectors that are independent of the matrix entries.} of $a_{ij}$ for all $i,j$, the pairs $((\vec A_N, \vec b_N))_{N \geq 1}$ satisfy a local law supported on $[-1,1]$ with
\begin{equation}\label{eq:GOElaw}
    \mu_\infty(\d x) = \mu_{\pU}(\d x) =  \frac{2}{\pi} \sqrt{1 - x^2} \d x.
\end{equation}
The most widely known case of a Wigner matrix is the so-called Gaussian Orthogonal Ensemble\footnote{Note that the scaling of the matrix here is chosen so that the eigenvalues typically lie within the interval $[-1,1]$.} (GOE):
\begin{equation}\label{eq:GOE}
\vec A_N = \frac{1}{2 \sqrt{2N}} (\vec{X} + \vec{X}^\T),
\end{equation}
where $X$ is an $N \times N$ Gaussian matrix with independent and identically distributed (iid) standard normal entries.  In the work of Trotter \cite{trotter_84} (see also \cite{dumitriu_edelman_02}) a full distributional characterization of the Householder tridiagonalization of GOE is given.  From this, one can see that the upper-left subblocks of this tridiagonalization tend to
\begin{equation*}
    \operatorname{tridiag}
    \left(\hspace{-1em} \begin{array}{c}
        \begin{array}{XXXX} 1/2 & 1/2 & \cdots & 1/2 \end{array} \\
        \begin{array}{XXXXX} \makebox[\widthof{1/2}][c]{0} & \makebox[\widthof{1/2}][c]{0}  & \cdots& \cdots & \makebox[\widthof{1/2}][c]{0}  \end{array} \\
        \begin{array}{XXXX} 1/2 & 1/2 & \cdots & 1/2 \end{array} 
    \end{array} \hspace{-1em}\right),
\end{equation*}
which correctly reflects the fact that the local law has \eqref{eq:GOElaw} as its limit\footnote{This is the Jacobi matrix associated with Chebyshev second-kind polynomials.}.

Then, for $\rho(x) = \frac{2}{\pi} \sqrt{1-x^2}$, $\pU_{n}(x) = 2^n \ppi_n(x;\mu_{\pU})$,
\begin{equation*}
    1 = \| \pU_{n} \|_{L^2(\mu_{\pU})} = 2^n \| \ppi_n(\cdot\,;\mu_{\pU})\|_{L^2(\mu_{\pU})}
    ,\qquad
    \gamma_{n-1}{(\mu_{\pU})}
    = -\ii \pi 2^{2n-1}, \qquad \mathfrak c = \frac{1}{2},
\end{equation*}
and
\begin{equation*}
    \vec{M}_n(z;\mu_{\pU})
    = 
    \begin{bmatrix}
    -\pi \ii \pU_{n}(z) \pU_{n-1}(z) & \pU_{n}(z)^2 \\ 
    \frac{\pi^2}{4} \pU_{n-1}(z)^2 &  -\pi \ii \pU_{n}(z) \pU_{n-1}(z)
    \end{bmatrix}.
\end{equation*}
It follows that
\begin{equation*}
    |\pU_{n}(z) \sqrt{1-z^2}|\leq 1
    ,\qquad
    z\in[-1,1],
\end{equation*}
so 
\begin{equation*}
   | \rho(z) \vec{M}_n(z;\mu_{\pU})|
   \underset{\mathrm{entrywise}}{\leq}
   \begin{bmatrix}
   2|\pU_{n-1}(z)| & \frac{2}{\pi}|\pU_{n}(z)| \\
   \frac{\pi}{2} | \pU_{n-1}(z) | & 2|\pU_{n-1}(z)|
   \end{bmatrix}, \quad z \in [-1,1].
\end{equation*}
Since $\|\pU_{n}\|_{[-1,1]} \leq n+1$, we see that we may set $\delta = 1, \delta' = 2$ in \cref{thm:forwards}. Then for \cref{asm:nearby_regular} to hold with high probability, we take $\gamma = 3/2$, $\beta < 1/3$.  This gives forward stability, with high probability, provided
\begin{equation}\label{bounds}
    k = o(N^{\frac{2}{54} - \epsilon}), \quad  \epslan = o(k^{-11}),
\end{equation}
with forward error $O(\epslan k^{11})$. Better bounds could be obtained by deriving a version of \cref{thm:forwards} which uses an explicit bound on the growth of the orthogonal polynomials of $\mu_U$ instead of \cref{thm:regularity_bounds_fwd}.

\subsection{Sample covariance matrices}
\label{sec:sample_covariance}

Consider the random matrix
\begin{equation}\label{eqn:Wishart}
\vec A_N = \frac{1}{N} \vec X \vec X^\T, \quad  \vec X = (x_{ij})_{\substack{1 \leq i \leq N\\ 1 \leq j \leq M}}, \qquad M \geq N,
\end{equation}
where the real-valued random variables $x_{ij}$ are jointly independent for all $i,j$ and satisfy
\begin{equation*}
    \mathbb E [x_{ij}]= 0, \qquad \mathbb E [x_{ij}^2]= 1,
\end{equation*}
and
\begin{equation*}
    \mathbb E [|x_{ij}|^k] \leq C_k < \infty,
\end{equation*}
for all $k > 2$. Such a matrix is called a sample covariance matrix. In \cite{knowles_yin_17}, for example, it is shown that for any sequence of vectors $(\vec b_N)_{N \geq 1}$ that are independent\footnote{As noted above, the reference here again establishes this result for a deterministic sequence of vectors.} of $x_{ij}$ for all $i,j$, the pairs $((\vec A_N, \vec b_N))_{N \geq 1}$ satisfy a local law supported on $[\lambda_-,\lambda_+]$ with
\begin{equation*}
    \mu_\infty(\d x) = \mu_{\rm{MP}}(\d x) = \frac{1}{2 \pi d x} \sqrt{(\lambda_+ - x)(x- \lambda_-)} \,\d x
    , \qquad 
    \lambda_\pm = (1 \pm \sqrt{d})^2,
\end{equation*}
if $N/M \to d \in (1,\infty)$.  This is the well-known Marchenko--Pastur law.

The most widely studied example of a sample covariance matrix is the so-called Wishart distribution \cite{wishart_28} where $x_{ij}$ are iid standard normal random variables and in this case the Golub-Kahan bidiagonalization procedure can be carried out in a distributional sense \cite{silverstein_86} (see also \cite{dumitriu_edelman_02}).  From this, one sees that the upper-left subblocks of this tridiagonalization tend to
\begin{equation*}
    \operatorname{tridiag}
    \left(\hspace{-1.5em} \begin{array}{c}
        \begin{array}{XXXX} \sqrt{d}  & \sqrt{d} & \cdots & \sqrt{d}  \end{array} \\
        \begin{array}{XXXXX} 1 & 1 +d & \cdots& \cdots & 1+d \end{array} \\
        \begin{array}{XXXX} \sqrt{d}   &\sqrt{d}   & \cdots & \sqrt{d}   \end{array} 
    \end{array} \hspace{-.75em}\right).
\end{equation*}
This leads one to conjecture that this gives three-term recurrence coefficients for the polynomials orthogonal to the Marchenko--Pastur distribution \cite{marcenko_pastur_67} (in fact, this provides an alternate proof of this).  Consider $q_n(x)$, $n = 0,1,2$ that satisfy, $q_0(x) = 1$,
\begin{align*}
x q_0(x) &= \sqrt{d}\,q_1(x) +  q_0(x) ,\\
x q_n(x) &= \sqrt{d}\, q_{n+1}(x) + (1 + d) q_n(x) + \sqrt{d}\, q_{n-1}(x), \qquad n \geq 1.
\end{align*}    
Next, note that
\begin{equation*}
    y\mapsto  2 y \sqrt{d} +1  + d
\end{equation*}
maps $[-1,1]$ to $[\lambda_-,\lambda_+]$.  Set $x = 2 y \sqrt{d} +1  + d$, $\check q_n (y) = q_n(2 y \sqrt{d} +1  + d)$ and we find
\begin{equation*}
\check q_1(y)  = (2 y + \sqrt d) \check q_0(y)
,\qquad
 \check q_{n+1}(y) + \check q_{n-1}(y)  = 2 y  \check q_n(y)
 , \qquad n \geq 1.
\end{equation*}
From this it follows that
\begin{equation*}
\check q_0(y) = U_0(y), \quad \check q_1(y) = U_1(y) + \sqrt{d}\, U_0(y),
\end{equation*}
and therefore, if we used the convention that $U_{-1}(y) = 0$
\begin{equation*}
    \check q_n(y) = U_n(y) + \sqrt{d}\, U_{n-1}(y), \qquad n \geq 0.
\end{equation*}
By explicitly calculating inner products, it was shown in \cite{deift_trogdon_20} that 
\[p_n(x;\mu_{\rm{MP}}) = \check q_n\left( \frac{ x - 1 - d}{2 \sqrt{d}} \right) = q_n(x).\]
It can also be shown that $ \mathfrak c(\lambda_+,\lambda_-) = \sqrt{d}$ is such that $\mathfrak c^n \ppi_n(x;\mu_{\rm{MP}}) = p_n(x;\mu_{\rm{MP}})$.  Therefore we obtain similar bounds on $\vec M_n(x,\mu_{\rm{MP}})$ as we did for $\vec M_n(x,\mu_{\pU})$.  Then \eqref{bounds} holds in the same way as for Wigner matrices.

\section{Examples}

In this section, we provide numerical experiments for several examples to which our analysis can be applied.

\subsection{Wigner matrices}
\label{sec:experiments:Wigner}

\begin{figure}[htb]
    \centering
    \begin{subfigure}[t]{0.48\textwidth}
         \centering
         \includegraphics[scale=.7]{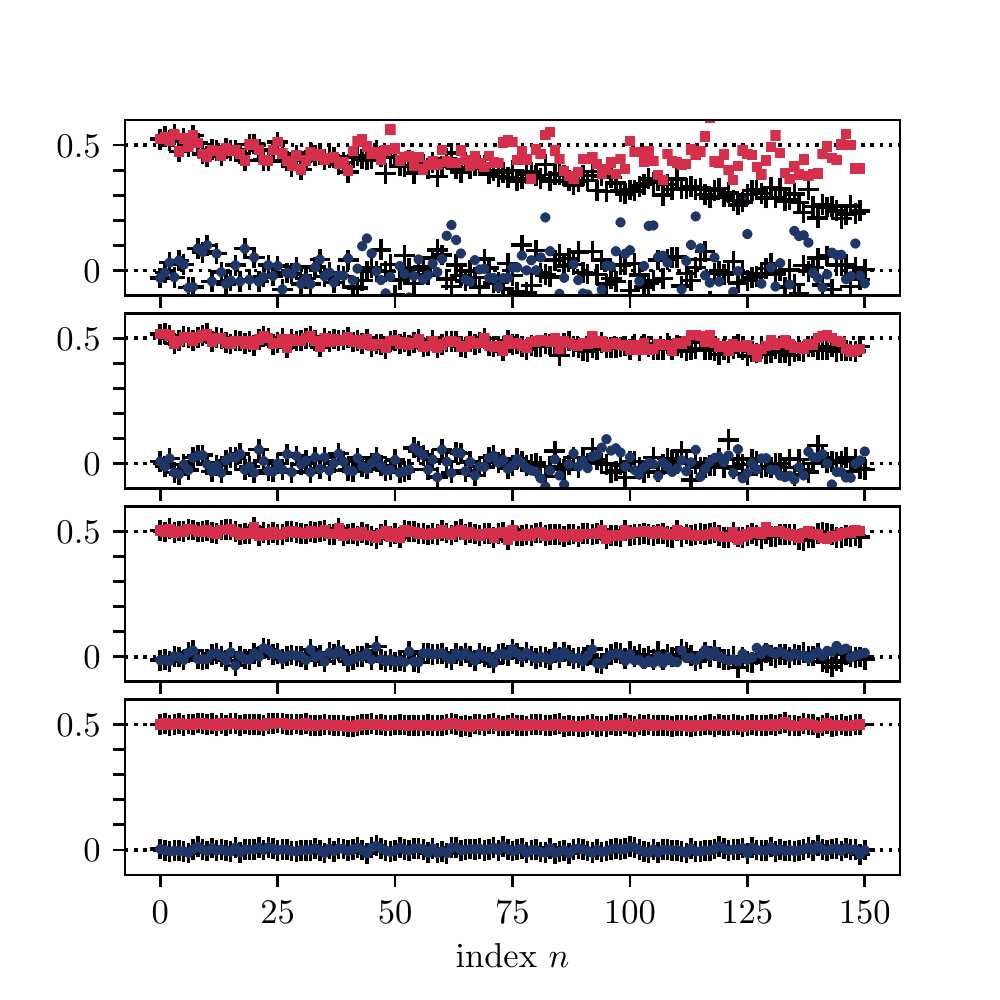}
         \caption{Recurrence coefficients $\overline{\alpha}_n$ \textup{(\includegraphics[scale=.8]{imgs/legend/circ.pdf})} and $\overline{\beta}_n$ \textup{(\includegraphics[scale=.8]{imgs/legend/square_red.pdf})}. Exact arithmetic counterparts shown as pluses \textup{(\includegraphics[scale=.8]{imgs/legend/plus_k.pdf})} and limiting values shown as dotted lines \textup{(\includegraphics[scale=.8]{imgs/legend/dot.pdf})}.}
         \label{fig:motivating_experiment_coeff_scale}
    \end{subfigure}
    \hfill
    \begin{subfigure}[t]{0.48\textwidth}
         \centering
         \includegraphics[scale=.7]{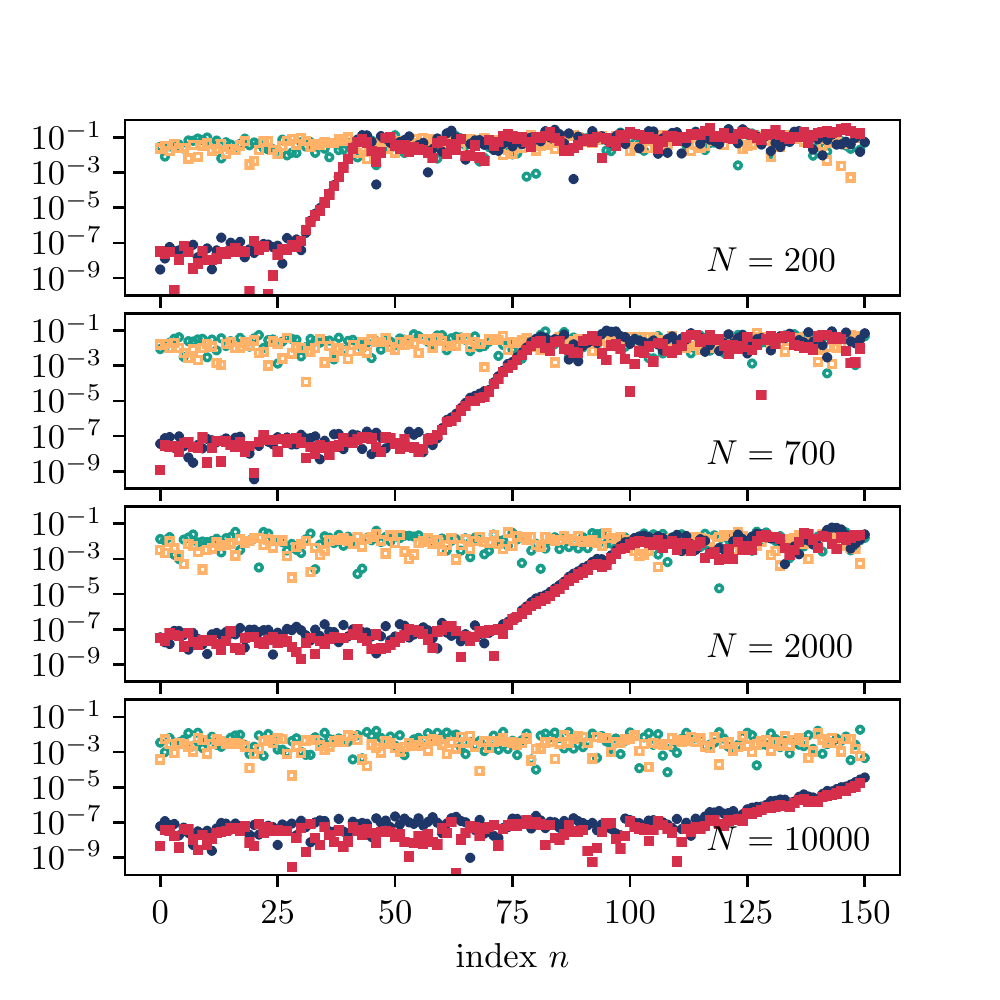}
         \caption{Forward error of recurrence coefficients $|\alpha_n - \overline{\alpha}_n|$ \textup{(\includegraphics[scale=.8]{imgs/legend/circ.pdf})} and $|\beta_n - \overline{\beta}_n|$ \textup{(\includegraphics[scale=.8]{imgs/legend/square_red.pdf})}
         and distance to limiting values $|0 - \overline{\alpha}_n|$ \textup{(\includegraphics[scale=.8]{imgs/legend/circ_green.pdf})} and $|1/2 - \overline{\beta}_n|$ \textup{(\includegraphics[scale=.8]{imgs/legend/square_yellow.pdf})}.}
         \label{fig:motivating_experiment_coeff_fe_scale}
    \end{subfigure}
    \caption{Output of Lanczos run on $(\vec{A}_N,\vec{b}_N)$ in single precision arithmetic, where $\vec{A}_N$ is a GOE matrix of size $N$ and $\vec{b}_N$ is an independent vector; see \cref{fig:motivating_experiment} for more details.}
    \label{fig:motivating_experiment_scale}
\end{figure}

In \cref{fig:motivating_experiment_scale}, we show plots akin to those in \cref{fig:motivating_experiment} for several values of $N$.
In particular, we take $\vec{A}_N$ as a random matrix from the Gaussian Orthogonal Ensemble \cref{eq:GOE} and $\vec{b}_N$ as a vector independent of $\vec{A}_N$.
As expected, as $N$ increases so that $\muN$ becomes nicer, the Lanczos algorithm remains forward stable for more iterations. The number of iterations for which it remains forward stable grows sublinearly with respect to $N$. 
When $n$ is small enough relative to $N$, the finite precision coefficients $\overline{\alpha}_n$ and $\overline{\beta}_n$ are much closer to their exact arithmetic counterparts than to the limiting values.
This behavior is suggested by \cref{thm:forwards}.

\subsection{Solving random linear systems}
\label{sec:random_system}
The mathematical behavior of a number of Krylov subspace methods used to solve systems involving random matrices have been studied rigorously.
Such algorithms include conjugate gradient and MINRES \cite{deift_trogdon_20,paquette_trogdon_22,ding_trogdon_21}, (accelerated) gradient descent \cite{paquette_merrienboer_paquette_pedregosa_22}, Neumann series iteration \cite{zhang_trogdon_22}, and GMRES \cite{chen_trogdon_greenbaum_23}.
The most basic result of these analyses is that the macroscopic behavior of the algorithms, such as the error at step $k$, often becomes nearly deterministic when the random matrix is sufficiently large. 
That is, the error at step $k$ converges to some fixed deterministic value when the random matrix becomes large.
Since Lanczos-based methods such as conjugate gradient and MINRES are, in general,  very susceptible to the impacts of floating point arithmetic \cite{greenbaum_97} one may wonder the extent to which analyses such as \cite{deift_trogdon_20,paquette_trogdon_22,ding_trogdon_21} hold in finite precision arithmetic.

\begin{figure}[ht]
    \centering
    \includegraphics[width=\textwidth]{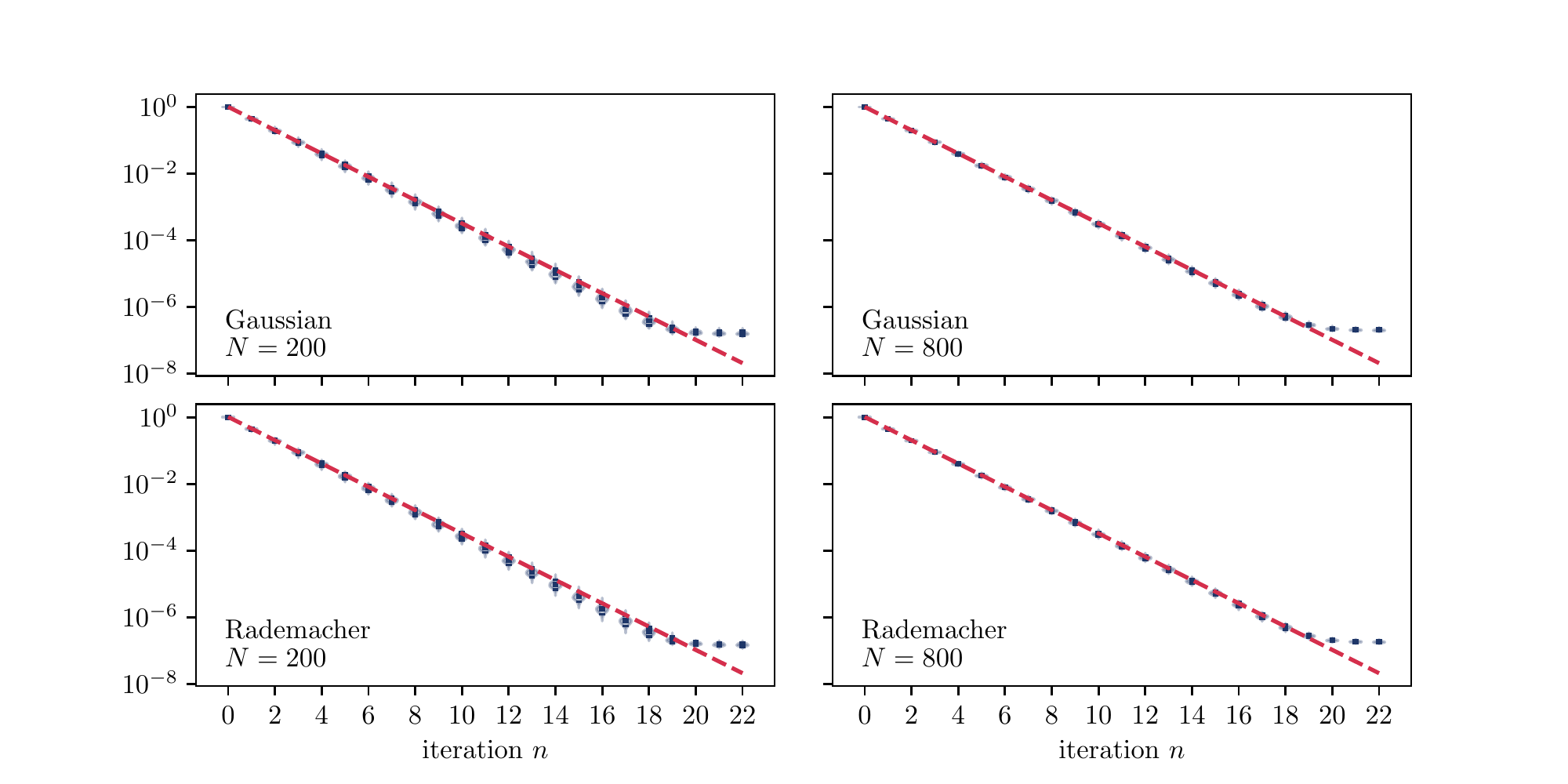}
    \caption{Error of Lanczos used to solve the system $\vec{A}\vec{x} = \vec{b}$ in single precision floating point arithmetic.
    Here  $\vec{b}_N$ is proportional to the all ones vector and $\vec{A}_N = N^{-1} \vec{X}\vec{X}^\T$ where the entries of $\vec{X}$ are either iid standard normal random variables or iid Rademacher random variables ($\pm 1$ with equal probability).
    For each $n,N$, the violin plot gives the distribution of error and the 5\%, 50\%, and 95\% quantiles are marked.
    Notice the convergence to the ``deterministic'' behavior \textup{(\includegraphics[scale=.8]{imgs/legend/dash_red.pdf})} as $N$ increases, at least until the maximal accuracy is reached. 
    }
    \label{fig:mp_system}
\end{figure}

In exact arithmetic, assuming $\vec{A}$ is positive definite, the iterate $\vec{x}_k := \Qk (\Tk)^{-1} \vec{e}_0$ is mathematically equivalent to the iterate produced by the conjugate gradient algorithm \cite{hestenes_stiefel_52} used to solve $\vec{A}\vec{x} = \vec{b}$.
This implies $\vec{x}_k$ is the optimal Krylov subspace approximation to $\vec{A}^{-1} \vec{b}$ in the $\vec{A}$-norm.
When $\vec{A}$ is a sample covariance matrix of the form described in \cref{sec:sample_covariance}, \cite{paquette_trogdon_22} shows that (under certain moment conditions),
\begin{equation}
\label{eqn:CG_error_limit}
    \| \vec{A}^{-1} \vec{b} - \vec{x}_k \|_{\vec{A}}
    \to \frac{d^{k/2}}{1-d}
    \text{ in probability as $N\to\infty$ and $N/M\to d$.}
\end{equation}

In finite precision arithmetic, the iterate 
\begin{equation}
\label{eqn:lanczos_CG}
    \overline{\vec{x}}_k := \Qkfp (\Tkfp)^{-1} \vec{e}_0
\end{equation}
is, in general, no longer optimal. 
However, the analysis in this paper can be applied to the iterate $\overline{\vec{x}}_k$, at least assuming it is computed exactly from the quantities $\Qkfp$ and $\Tkfp$.\footnote{A full analysis of the standard CG implementation \cite{hestenes_stiefel_52} is well beyond the scope of this paper. 
In fact, to the best of our knowledge, it is not even known rigorously whether Paige's analysis extends to such an implementation. 
In practice the finite precision arithmetic behavior of \cref{eqn:lanczos_CG} and the standard CG implementation are quite similar, so it is common to analyze \cref{eqn:lanczos_CG} \cite{greenbaum_89,strakos_greenbaum_92}.}

We perform a numerical experiment with sample covariance matrices $\vec{A}_N = N^{-1} \vec{X}\vec{X}^\T$, where the entries of $\vec{X}$ are either iid standard normal random variables or iid Rademacher random variables.
In particular, we generate matrices for each of these distributions at $d=0.2$ and $N=200$ or $N=800$ and set $\vec{b}_N$ proportional to the all ones vector.
We then run Lanczos on $(\vec{A}_N,\vec{b}_N)$ in single precision arithmetic to get $\Qkfp$ and $\Tkfp$.
We compute $\overline{\vec{x}}_k$ using a standard linear system solver from numpy in double precision arithmetic. 

The results of 1000 repetitions of each of these experiments are reported in \cref{fig:mp_system}.
As expected, at least until convergence stagnates, the error of the algorithm concentrates around the estimate \cref{eqn:CG_error_limit} as $N$ increases.
The error $\|\vec{A}^{-1}\vec{b} - \overline{\vec{x}}_k\|_{\vec{A}}$ stagnates around the machine precision, which is essentially all we could hope for given that $\Tkfp$ is computed in single precision floating point arithmetic.

\section{Deferred proofs}

\subsection{Proofs of bounds for regular measures}
\label{sec:regular_proofs}

\begin{proof}[Proof of \cref{thm:op_bound}]
    The proof follows \cite[Lemma 6]{fischer_96}.
    Fix $n\leq 2k-1$.
    As a special case of the Markov brothers' inequality for polynomials \cite[Theorem 1.10]{rivlin_81}, the derivative $\opp{n}(\cdot\,;\mu)$ of $\op{n}(\cdot\,;\mu)$ satisfies
    \begin{equation*}
        \|\opp{n}(\cdot\,;\mu)\|_{[\tta,\ttb]} 
        \leq \frac{2n^2}{\ttb-\tta} \|\op{n}(\cdot\,;\mu)\|_{[\tta,\ttb]}.
    \end{equation*}
    Let $x^*$ be such that $|\op{n}(x^*;\mu)| = \| \op{n}(\cdot\,;\mu) \|_{[\tta,\ttb]}$ and define
    \begin{equation*}
        \mathcal{B} := \Big\{ x \in [\tta,\ttb] : |x-x^*| \leq \frac{\ttb-\tta}{4n^2} \Big\}.
    \end{equation*}
    Using the triangle inequality, for any $x\in \mathcal{B}$,
    \begin{align*}
        |\op{n}(x;\mu)|
        &= \Big| \op{n}(x^*;\mu) + \int_{x^*}^{x} \opp{n}(y;\mu) \,\d{y} \Big|
        \\&\geq |\op{n}(x^*;\mu)| - |x-x^*| \| \opp{n}(\cdot\,;\mu) \|_{[\tta,\ttb]}
        \\&\geq |\op{n}(x^*;\mu)| - \frac{\ttb-\tta}{4n^2} \frac{2n^2}{\ttb-\tta} |\op{n}(x^*;\mu)|
        \\&\geq \frac{1}{2}|\op{n}(x^*;\mu)|.
    \end{align*}
    Both endpoints of $\mathcal{B}$ must be in $[\tta,\ttb]$, so $\max(\mathcal{B})-\min(\mathcal{B}) \geq (\ttb-\tta)/(4n^2) \geq (\ttb-\tta)/(16k^2)$ and hence $\mu(\mathcal{B}) \geq K$.
    Using this and the fact that $\op{n}(\cdot\,;\mu)$ is normalized,
    \begin{equation*}
        1 = \int_{\tta}^{\ttb} \op{n}(x;\mu)^2 \mu(\d{x})
        \geq \int_{\mathcal{B}} \op{n}(x;\mu)^2 \mu(\d{x})
        \geq \frac{1}{4} \op{n}(x^*;\mu)^2 \mu(\mathcal{B})
        \geq \frac{1}{4} \op{n}(x^*;\mu)^2 K.
    \end{equation*}
    Solving for $|\op{n}(x^*;\mu)|$ we find
    \[
        \| \op{n}(\cdot\,;\mu) \|_{[\tta,\ttb]}
        = |\op{n}(x^*;\mu)|
        \leq \frac{2}{\sqrt{K}}.
    \]
    Since this holds for all $n\leq 2k-1$, the lemma follows.
\end{proof}

\begin{proof}[Proof of \cref{thm:regularity_bounds}]

Using the triangle inequality, for any $x,y$,
\begin{align*}
    | \muN([x,y]) - \mu_{\infty}([x,y]) |
    &= |\muN((-\infty,y])-\muN((-\infty,x)) - \mu_{\infty}((-\infty,y])+\mu_{\infty}((-\infty,x)) |
    \\&\leq | \muN((-\infty,y]) - \mu_{\infty}((-\infty,y])  | + |\muN((-\infty,x)) - \mu_{\infty}((-\infty,x)) |
    \\&\leq 2 \KS(\muN, \mu_{\infty}).
\end{align*}
Thus, one easily verifies that
\begin{equation*}
    \muN([x,y]) \geq \mu_{\infty}([x,y]) - 2\KS(\muN, \mu_{\infty}).
\end{equation*}
Suppose $x,y  \in [\tta',\ttb']$ and $|x-y| \geq ({\ttb-\tta})/({16k^2})$.
Then the length of $[x,y]\cap[\tta,\ttb]$ is at least $(\ttb-\tta)/(32k^2)$, so using \cref{eqn:asm_regular}, 
\[
\mu_\infty([x,y]) \geq L\left(\frac{\ttb-\tta}{32k^2}\right)^\gamma.
\]
Using this, our assumption on $k$, \cref{eqn:KS_assum}, and the fact $\ttb'
-\tta' \geq \ttb-\tta$, we obtain a bound
\begin{equation*}
    \muN([x,y]) 
    \geq L\left(\frac{\ttb-\tta}{32k^2}\right)^\gamma - 2N^{-\alpha}
    \geq  N^{-\alpha}
    ,\qquad \forall x,y  \in [\tta',\ttb'] : |x-y| \geq ({\ttb'-\tta'})/({16k^2}).
\end{equation*}
\Cref{thm:op_bound} with $K=N^{-\alpha}$ then gives a bound for the orthogonal polynomials $\op{n}$ for $\mu = \muN$:
\begin{equation*}
    \opMax{k}(\muN;[\tta',\ttb']) 
    \leq \frac{2}{\sqrt{N^{-\alpha}}}
    \leq 2\sqrt{\frac{3}{L}} \left(\frac{32}{\ttb-\tta}\right)^{\gamma/2} k^\gamma.
\end{equation*}
Clearly $2\sqrt{3} < 4$, so the result is established.
\end{proof}

\begin{proof}[Proof of \cref{thm:regularity_bounds_fwd}]
Fix $n\leq 2k-1$.
Given \cref{eqn:asm_regular}, we can apply \cref{thm:op_bound} with $K = L ((\ttb-\tta)/(16k^2))^{\gamma}$ to get a bound
\begin{equation*}
    \opMax{k}(\mu_\infty;[\tta,\ttb])
    \leq \frac{2}{\sqrt{L}} \left( \frac{16k^2}{\ttb-\tta} \right)^{\gamma/2}.
\end{equation*}
This gives the first part of the lemma.

The Markov brothers' inequality for polynomials implies
\begin{equation*}
     \|\opp{n}(\cdot\,;\mu_\infty)\|_{[\tta',\ttb']}
     \leq \frac{2 n^2}{\ttb'-\tta'} \|\op{n}(\cdot\,;\mu_\infty)\|_{[\tta',\ttb']}.
\end{equation*}
Thus, since the supports of $\muN$ and $\mu_\infty$ are contained in $[\tta',\ttb']$, integrating by parts,
\begin{align*}
    |\mom{n}(\muN;\mu_\infty) - \mom{n}(\mu_{\infty};\mu_\infty)|
    &=\bigg|\int_{\tta'}^{\ttb'} \op{n}(x;\mu_\infty) \big(\mu_{N}(\d{x}) - \mu_{\infty}(\d{x})\big) \bigg|
    \\&\leq \int_{\tta'}^{\ttb'} |\opp{n}(x;\mu_\infty)| \big|\mu_{N}((-\infty,x]) - \mu_{\infty}((-\infty,x])\big|\,\d{x}
    \\&\leq  (\ttb'-\tta') \| \opp{n}(\cdot\,;\mu_\infty) \|_{[\tta',\ttb']} \KS(\muN, \mu_{\infty}) 
    \\&\leq 2 n^2 \|\op{n}(\cdot\,;\mu_\infty)\|_{[\tta',\ttb']}  \KS(\muN, \mu_{\infty}).
\end{align*}
Since $n\leq 2k-1$, with $\eta = 1/(16k^2) \leq 1/(2(2k)^2) \leq 1/(2n^2)$, \cref{thm:poly_bd} (which is independent of this result) yields the bound 
\begin{align*}
    \|\op{n}(\cdot\,;\mu_\infty)\|_{[\tta',\ttb']}
    &= \max_{x\in[-1-\eta,1+\eta]} \left| \op{n} \left(\frac{\tta+\ttb}{2}+\frac{\ttb-\tta}{2}x ;\mu_\infty\right)\right|
    \\&\leq 2 \max_{x\in[-1,1]} \left| \op{n} \left( \frac{\tta+\ttb}{2}+\frac{\ttb-\tta}{2}x ;\mu_\infty\right)\right|
    \\&= 2 \| p_n(\cdot\,;\mu_\infty) \|_{[\tta,\ttb]} 
    \leq 2 \opMax{k}(\mu_\infty;[\tta,\ttb]).
\end{align*}

Since $n \leq 2k$, using \cref{eqn:KS_assum} and our bound on $\opMax{k}(\mu_\infty;[\tta,\ttb])$, we obtain a bound for the modified moments
\begin{align*}
    |\mom{n}(\muN;\mu_\infty) - \mom{n}(\mu_{\infty};\mu_\infty)|
    &\leq 2 (2k)^2  2\opMax{k}(\mu_\infty;[\tta,\ttb]) \KS(\muN, \mu_{\infty})
    \\&\leq  16 k^{2} \frac{2}{\sqrt{L}}  \left(\frac{16k^2}{\ttb-\tta}\right)^{\gamma/2} N^{-\alpha}.
\end{align*}
Thus, using our assumption on $k$ and that $n\leq 2k-1$,
\begin{equation*}
    \momDel{n}(\muN,\mu_\infty;\mu_\infty)
    =\max_{n\leq 2k-1} |\mom{n}(\muN;\mu_\infty) - \mom{n}(\mu_{\infty};\mu_\infty)|
    \leq c.
\end{equation*}
The lemma is established.
\end{proof}

\subsection{Other proofs}
\label{sec:thm:auxiliary_proofs}

\begin{proof}[Proof of \cref{thm:poly_bd}]
It is well-known that for any $x \in \mathbb{R}\setminus[-1,1]$,
\[
|p(x)|
\leq |T_n(x)|
,\qquad \forall p : \deg(p) \leq n, \|p\|_{[-1,1]} \leq 1.
\]
Thus, it suffices to show 
\begin{equation*}
    |T_n(x)| \leq 2,
    \qquad x\in[-1-1/(2n^2), 1+1/(2n^2)].
\end{equation*}
We will in fact show $|T_n(z)| \leq 2$ for all $z\in E$, where 
$E := \{ (w + w^{-1})/2 : 1\leq |w| \leq r + \sqrt{r^2-1} \}$ is the Bernstein ellipse with rightmost point $r=1+1/(2n^2)$. 

Suppose $J(w) = \frac{1}{2}(w+w^{-1})$.
Then, it is well known that, for any $n\geq 0$,
\begin{equation*}
    T_n(J(w)) = \frac{1}{2}\left( w^n + w^{-n} \right).
\end{equation*}
Let $z=J(w)$ be an arbitrary point on the boundary of the Bernstein ellipse $E$. 
Set $\rho = |w|$ so
\begin{equation*}
    |T_n(z)| \leq \frac{1}{2} \left( \rho^n + \rho^{-n} \right).
\end{equation*}
Suppose $n\geq 2$ and let $\rho = \rho(n) = 1+\ln(2+\sqrt{3})/n$.
Then, \begin{equation*}
    \lim_{n\to\infty}
    \frac{1}{2} \left(\rho(n)^{n} + \rho(n)^{-n}\right)
    = 2.
\end{equation*}
By direct computation, one verifies that
\begin{align*}    \frac{\d}{\d{n}} \left(\rho(n)^n + \rho(n)^{-n}\right)
    &= \left(\rho(n)^n - \rho(n)^{-n}\right)\left(\ln(\rho(n)) - 1 + \rho(n)^{-1} \right).
\end{align*}
We always have
\begin{equation*}    
    \ln(\rho(n)) - 1 + \rho(n)^{-1}
    > 0
    ,\qquad \rho(n) > 1.
\end{equation*}
Since $\rho(n)^n$ is monotonically increasing with $n$ and $\rho(n)^{-n}$ is monotonically decreasing with $n$, 
\begin{equation*}
    \rho(n)^{n} - \rho(n)^{-n}
    \geq \bigg(1 + \frac{\ln(2+\sqrt{3})}{2}\bigg)^{2} - \bigg(1 + \frac{\ln(2+\sqrt{3})}{2}\bigg)^{-2}
    > 0.
\end{equation*}
Thus, the convergence of $(\rho(n)^n + \rho(n)^{-n})/2$ to $2$ is monotonic from below. 
This implies  $|T_n(z)|\leq 2$ for $z\in E_\rho$ and $n\geq 2$.

Now, note that 
\begin{align*}
\label{eqn:Erho_zbd}
    \frac{1}{2} \left(\rho(n)+\rho(n)^{-1} \right)
    &= 1 + \frac{1}{2n^2}\bigg( \ln(2+\sqrt{3})^2 - \frac{\ln(2+\sqrt{3})^3}{\ln(2+\sqrt{3})+n} \bigg)
    \\&\geq 1 + \frac{1}{2n^2}\bigg( \ln(2+\sqrt{3})^2 - \frac{\ln(2+\sqrt{3})^3}{\ln(2+\sqrt{3})+2} \bigg)
    \geq 1 + \frac{1}{2n^2}.
\end{align*}
This implies $E\subseteq E_\rho$ for $n\geq 2$.
Clearly $\pT_0(z) = 1 \leq 2$ and $ |\pT_1(z)| =  |z| \leq 2$ for all $z\in E$.
Thus, for all $n \geq 0$,
\begin{equation*}
    |T_n(z)|
    \leq 2
    ,\qquad z\in E.
\end{equation*}
The result follows since $[-1-1/(2n^2), 1+1/(2n^2)] \subseteq E$.
\end{proof}

\begin{proof}[Proof of \cref{thm:assoc_p}]
Suppose the lemma holds for $i<n$. 
Then, 
\begin{align*}
    d_n(x) &= 2x d_{n-1}(x) - d_{n-2}(x) + 2 f_{n-1}
    \\&= 2 x \Bigg( U_{n-2}(x) f_0 + 2 \sum_{i=2}^{n-1} U_{n-1-i}(x) f_{i-1} \Bigg) 
     - \Bigg( U_{n-3}(x) f_0 + 2 \sum_{i=2}^{n-2} U_{n-2-i}(x) f_{i-1} \Bigg) + 2f_{n-1}
    \\&= \big( 2x U_{n-2}(x) - U_{n-3}(x) \big)f_0 + 2 \Bigg(\sum_{i=2}^{n-2} \big(2x U_{n-1-i}(x) - U_{n-2-i}(x) \big)f_{i-1} \Bigg)  
    + 4xf_{n-2} + 2f_{n-1}
    \\&= U_{n-1}(x) f_0 + 2 \Bigg(\sum_{i=2}^{n-2} U_{n-i}(x) f_{i-1} \Bigg) + 2 U_1(x) f_{n-2} + 2U_0(x) f_{n-1}
    \\&= U_{n-1}(x) f_0 + 2 \sum_{i=2}^{n} U_{n-i}(x) f_{i-1}.
\end{align*}
The result follows as the base case is assumed.
\end{proof}

\printbibliography
\end{document}